\documentclass[12pt]{amsart}
\usepackage{amssymb,enumerate,mathrsfs, amsmath, amsthm, bbm, graphicx,amsaddr}
\usepackage{url,titletoc,enumitem}
\usepackage{tikz-cd}
\usepackage{bbold}
\usepackage{makecell}

\DeclareMathOperator{\Aut}{Aut}

\DeclareMathOperator{\denom}{denom}
\DeclareMathOperator{\num}{num}
\DeclareMathOperator{\perm}{perm}
\DeclareMathOperator{\Hom}{Hom}
\DeclareMathOperator{\new}{new}
\DeclareMathOperator{\rel}{rel}
\DeclareMathOperator{\base}{base}
\DeclareMathOperator{\inc}{inc}

\newtheorem{theorem}{Theorem}[section]
\newtheorem{cor}[theorem]{Corollary}
\newtheorem{lemma}[theorem]{Lemma}

\newtheorem{prop}[theorem]{Proposition}

\newtheorem{conj}[theorem]{Conjecture}

\newtheorem{example}[theorem]{Example}

\newtheorem{lem}[theorem]{Lemma}

\newtheorem*{lemma*}{Lemma}

\numberwithin{equation}{section}

\def\bbQ{ {\mathbb Q}}
\def\bbZ{ {\mathbb Z}}

\def\bb1{ {\mathbb 1}}
\def\bbE{ {\mathbb E}}

\def\fkm{ {\mathfrak m}}

\def\cO{ {\mathcal O} }
\def\cP{{\mathcal P}}

\usepackage{biblatex}
\usepackage{hyperref}
\makeindex

\title{DENSITY OF $p$-ADIC POLYNOMIALS GENERATING EXTENSIONS WITH FIXED SPLITTING TYPE}

\author{John Yin$^1$}

\address{$^1$Department of Mathematics at University of Wisconsin, Madison}
\email{jbyin@wisc.edu}
\smallskip

\begin{document}
\begin{abstract}
We prove that the density of polynomials $P(x)=\sum_{i=0}^n a_n x^n$ over a local field $K$ generating an \'etale extension with specified splitting type is a rational function in terms of the size of the residue field of $K$ in the case where the splitting type is tame. Moreover, we give a computable recursive formula for these densities and compute the asymptotics of this density as the size of the residue field tends to infinity.
\end{abstract}
\maketitle
\sloppy
\section{Introduction}
In \cite{bhargava_cremona_fisher_gajović_2022}, the authors conjectured that the probability that a random $p$-adic polynomial $P(x)= c_nx^n + c_{n-1}x^{n-1} \dots + c_0$ with coefficients in $\bbZ_p$ generates an \'etale extension such that $p$ splits in a specified way (e.g. $\bbQ_p[x]/f$ is a totally ramified extension, unramified extension, a product of two extensions with specified residue degrees) is a rational function $\rho$ of $p$, and moreover satisfies the functional equation $\rho(p)=\rho(1/p)$. We will prove the rationality of $\rho$ in the tamely ramified case. As we shall see, our methods are more general and we will prove rationality of the density of polynomials satisfying many other conditions. Moreover, we will give a recursive method to calculate these densities.

We parametrize polynomials $P(x)$ of degree at most $n$ as above by $\mathbb Z_p^{n+1}$ with the Haar measure normalized so that the total measure of $\bbZ_p^{n+1}$ is $1$. Now given an \'etale extension $L$, it is isomorphic to a product of field extensions $L_i$, so that $L = \prod_{i=1}^m L_i$. We define the splitting type of $L$ to be the multiset of pairs $$\sigma = \{(e(L_1),f(L_1)),(e(L_2),f(L_2)), \dots, (e(L_n),f(L_n))\}.$$ This is written as $(f(L_1)^{e(L_1)}, \dots, f(L_m)^{e(L_m)})$ in \cite{bhargava_cremona_fisher_gajović_2022}, and we will follow their notation. Fixing such a splitting type $\sigma$, let $U_\sigma(p) \subset \mathbb Z_p^{n+1}$ be the $p$-adic open subset of polynomials with splitting type $\sigma$ and $\rho(n,\sigma;p) = \mu_{\mathrm{Haar}}(U_\sigma(p))$. As a few sample examples, \cite{bhargava_cremona_fisher_gajović_2022} computes
\begin{align}\label{bhargava_computations}
    \rho(2,(1^11^1);p) &= \frac{1}{2} &\rho(2,(2^1);p) &= \frac{p^2-p+1}{2(p^2+p+1)}\\
    \rho(3,(1^11^11^1);p) &= \frac{p^4+2p^2+1}{6(p^4+p^3+p^2+p+1)} &\rho(3,(1^12^1);p) &= \frac{p^4+1}{2(p^4+p^3+p^2+p+1)}.
\end{align}
Similarly, let $\alpha(n,\sigma;p)$ and $\beta(n,\sigma;p)$ be defined to be the measure of $\{P \in U_{\sigma(p)}: P \text{ is monic}\}$ and $\{P \in U_{\sigma(p)}: P \text{ is monic and } P \equiv x^n \mod p\}$, respectively. Based on numerical evidence and some of their results, \cite[Conjecture 1.2]{bhargava_cremona_fisher_gajović_2022} states the following conjecture.
\begin{conj}\label{conj: intro bhargava}
The densities $\rho(n,\sigma;p)$, $\alpha(n,\sigma;p)$, and $\beta(n,\sigma;p)$ are rational functions in $p$ and satisfy
\[\rho(n,\sigma;p^{-1}) = \rho(n,\sigma;p).\]
\[\alpha(n,\sigma;p^{-1}) = \beta(n,\sigma;p).\]

\end{conj}
In this paper, we will prove the rationality part of this conjecture for tame extensions (\ref{thm:main}).

\subsection{Overview of proof}
One main ingredient is \ref{thm:Xavier_Integral}, proven in \cite{caruso2022zeroes}, which makes use of the $p$-adic Kac-Rice formula to compute the average number of roots of a random polynomial in a fixed open subset of $\cO_L$ for a given \'etale extension $L$. To use this formula, it turns out one needs to calculate the density of elements $\alpha \in \cO_L$ whose minimal polynomial $P_{\alpha/\bbQ_p}$ over $\bbQ_p$ has a specified valuation of polynomial discriminant $\Delta_{P_{\alpha/\bbQ_p}}$. That is, we need to calculate the density $\mu_L(\{\alpha \in \cO_L: v_p(\Delta_{P_{\alpha/\bbQ_p}})=c\})$ for every $c$, where $\mu_L$ is the normalized Haar measure on $\cO_L$. Calculating this density and summing over the \'etale extensions $L$ is the main technical part of this paper, which we will explain below. With this we obtain the average number of roots of a random polynomial in extensions with prescribed splitting behavior. Finally, a couple of transformations are done in Sections \ref{sec:main} and \ref{sec:red_to_OL_thm_proof} to get from this to the density of random polynomials generating an extension with prescribed splitting behavior, which concludes the proof. 

As an example of the process involved in the calculation of $\mu_L(\{\alpha \in \cO_L: v_p(\Delta_{P_{\alpha/\bbQ_p}})=c\})$, we fix a square root $p^{1/2}$ of $p$ and take $L$ to be $\bbQ_p[\zeta_{p^3-1},p^{1/2}]$, which is an extension of $\bbQ_p$ of ramification index $2$ and inertia degree $3$. Also let $\fkm_L$ be the maximal ideal in $\cO_L$. Elements of $\alpha \in \cO_L$ have a Teichm\"uller expansion of the form $\alpha = \sum a_i p^{i/2}$, where $a_i$ is either a power of a $(p^3-1)$-th root of unity or $0$. The $6$ Galois conjugates of $\alpha$ are of the form $\sum a_i^{p^j}(-1)^{ib} p^{i/2}$ where $j \in \{0,1,2\}$ and $b \in \{0,1\}$. With this expression of the Galois conjugates, one can recursively calculate the valuation of the discriminant of the minimal polynomial by using the definition $\Delta_{P_{\alpha/\bbQ_p}} = \prod_{\alpha' \neq \alpha'' \text{ Galois conjugates of }\alpha}(\alpha'-\alpha'')$. Here we use the convention that in the case our polynomial is linear, the product is empty and our discriminant is $1$. This recursion is somewhat combinatorially complicated, but we will work out a few steps in the recursion for this $L$ to give the reader an idea of the process. 

The context of the following paragraphs is that we wish to calculate the $p$-adic valuation of the product of pairwise differences between Galois conjugates of $\alpha$ by analyzing its Teichm\"uller expansion $\alpha = \sum_{i=0}^\infty a_i p^{i/2}$. We start out with considering the lowest degree coefficient $a_0$. There are two cases. If $a_0 \in \bbQ_p$, which happens for exactly $p$ choices of $a_0$, then the coefficient of $p^{0/2}$ in the Teichm\"uller expansion of the Galois conjugates of $\alpha$ must all be $a_0$, and hence for these elements, $v_p(P_{\alpha/\bbQ_p}) = v_p(P_{\alpha-a_0})$, where $\alpha-a_0 \in \fkm_L$. So, we just found that $$\mu_L(\{\alpha \in \cO_L: a_0 \in \bbQ_p, v_p(\Delta_{P_{\alpha/\bbQ_p}})=c\})=p\mu_L(\{\alpha \in \fkm_L: v_p(\Delta_{P_{\alpha/\bbQ_p}})=c\}).$$

We now do the same in the other case. If $a_0 \notin \bbQ_p$, which happens for $p^3-p$ choices of $a_0$, then the coefficients of $p^{0/2}$ in the Teichm\"uller expansion of Galois conjugates of $\alpha$ are precisely the Galois conjugates of $a_0$, of which there will be $3$. In fact, it turns out that the $6$ Galois conjugates are neatly partitioned into $3$ sets of $2$ Galois conjugates, with each set containing the three Galois conjugates with a fixed $p^{0/2}$ coefficient. Now, $v_p(a_0^{p^{j'}}-a_0^{p^j})=0$ for distinct $j,j' \in \{0,1,2\}$ since $\alpha$ is a $p^3-1$-th root of unity. In other words, for Galois conjugates $\alpha', \alpha''$ in different partitions (i.e. with different $p^{0/2}$ coefficients), we have $v_p(\alpha'-\alpha'')=0$. So, in the context of finding the valuation of the product of pairwise differences of Galois conjugates of $\alpha$, all that's left is to calculate the valuation of the product of $\alpha'-\alpha''$, as $\alpha',\alpha''$ vary over pairs of Galois conjugates of $\alpha$ with the same $p^{0/2}$-th coefficient. There are $3$ such pairs, one coming from each element of the partition above. Now, note that the Galois group acts transitively on the partitions themselves. In other words, the valuation of the difference between each pair are the same. So, it suffices to study one set in the partition, say the set $\{\alpha,\alpha'\}$ whose coefficient of $p^{0/2}$ is $a_0$. Now we note that $\alpha,\alpha'$ are precisely the Galois conjugates of $\alpha$ over $\bbQ_p[a_0]$. Moreover, $\alpha-\alpha'=(\alpha-a_0)-(\alpha'-a_0)$, and $\alpha-a_0$ lies in $\fkm_L$. Thus, also we obtain $$\mu_L(\{\alpha \in \cO_L: a_0 \notin \bbQ_p, v_p(\Delta_{P_{\alpha/\bbQ_p}})=c\})=(p^3-p)\mu_L(\{\alpha \in \fkm_L: v_p(\Delta_{P_{\alpha/\bbQ_p[a_0]}})=c\}).$$

The above two paragraphs conclude the first step of the recursion. The combination of the two cases above reduced the calculation of $\mu_L(\{\alpha \in \cO_L: v_p(\Delta_{P_{\alpha/\bbQ_p}})=c\})$ to the calculation of multiple densities of the form $\mu_L(\{\alpha \in \fkm_L: v_p(\Delta_{P_{\alpha/K}})=c'\})$, for some extension $K$ of $\bbQ_p$ and integer $c' \leq c$. As we shall see, this pattern continues. Although in the above cases, $c'$ happened to equal $c$, we do one more step in the recursion to see that $c'$ can be less than $c$. We continue from the first case, where we ended up wanting to compute $\mu_L(\{\alpha \in \fkm_L: v_p(\Delta_{P_{\alpha/\bbQ_p}})=c\})$. Elements in $\alpha \in \fkm_L$ have Teichm\"uller expansion of the form $\sum_{i \geq 1} a_i p^{i/2}$. We again focus on the lowest degree term $a_1p^{1/2}$. Now there are three cases. The first case is if $a_1 = 0$, then we reduce to the calculation of $\{\alpha \in \fkm_L^2: a_1 = 0,  v_p(\Delta_{P_{\alpha/\bbQ_p}})=c\}$. The second case is if $a_1 \in \bbQ_p \setminus 0$, which occurs for a proportion of $p-1$ of all elements in $\fkm_L$, then the Galois conjugates are separated into $2$ partitions, each with $3$ conjugates with a common leading term in the Teichm\"uller expansion $\sum_{i \geq 1} a_i p^{i/2}$ is $a_1p^{1/2}$ or $-a_1p^{1/2}$. Now for conjugates $\alpha', \alpha''$ belonging to two different partitions, $v_p(\alpha'-\alpha'') = v_p(2a_1p^{1/2}) = 1/2$, assuming $p \neq 2$. This is one example of where the tameness assumption comes into play. There are $2 \cdot 3\cdot 3=18$ such ordered pairs, giving a contribution of $18/2$ to the $p$-adic valuation of the discriminant. Just as we did in the previous paragraph, we can now focus on product of pairwise differences of Galois conjugates $\alpha,\alpha'$ in one partition, say the one with leading term $a_1p^{1/2}$. Moreover, for Galois conjugates in this partition, $\alpha-\alpha'=(\alpha-a_1p^{1/2})-(\alpha'-a_1p^{1/2})$, and $\alpha-a_1p^{1/2}$ lies in $\fkm_L^2$. Now there are two partitions, with both of the product of the pairwise differences of elements in each partition giving the same contribution to the $p$-adic valuation of the discriminant. Thus, we have that $$\mu_L(\{\alpha \in \fkm_L: a_1 \in \bbQ_p \setminus \{0\},  v_p(\Delta_{P_{\alpha/\bbQ_p}})=c\})=(p-1)\mu_L(\{\alpha \in \fkm_L^2: v_p(\Delta_{P_{\alpha/\bbQ_p[a_1p^{1/2}]}})=(c-9)/2\}).$$ Finally, we will not work out in full detail the third case, but one can show with basic theory of local fields that if $a_1 \notin \bbQ_p$, then $a_1p^{1/2}$ will have $6$ conjugates, and in particular generate $L$ itself. So, the pairwise difference of Galois conjugates of $\alpha$ would all have $p$-adic valuation of $1/2$, and we would actually obtain $$\mu_L(\{\alpha \in \fkm_L: a_1 \notin \bbQ_p,  v_p(\Delta_{P_{\alpha/\bbQ_p}})=c\})=(p^3-p)\mu_L(\{\alpha \in \fkm_L^2: v_p(\Delta_{P_{\alpha/L}})=(c-15)/6\}).$$ This concludes another step of the recursion. Note that this point, we have reached a base case. Here we use the convention that the discriminant of a linear polynomial is $1$, and so $\mu_L(\{\alpha \in \fkm_L^2: v_p(\Delta_{P_{\alpha/L}})=(c-15)/6\})$ is $p^{-6}$ if $(c-15)/6 = 0$, and $0$ otherwise. 

More generally, given an integer $i$, a local field $K$, and an integer $c$, and $L$ an extension of $K$, our recursion reduces the calculation of $\mu_L(\{\alpha \in \fkm_L^{i}: v_p(\Delta_{P_{\alpha/K}})=c\})$ to $\mu_L(\{\alpha \in \fkm_L^{i+1}: v_p(\Delta_{P_{\alpha/F}})=c'\})$ for some extension $F/K$ contained in $L$ and $c'$ an integer at most $c$. Now consider the poset $$\cP = \{(i,c,F): i \in \bbZ_{\geq 0}, c \in \bbZ, K \subset F \subset L\}$$ with ordering given by $(i',c',F') \leq (i,c,F)$ if $i' \geq i$, $c' \leq c$, and $F' \supset F$. Associating the set $\{\alpha \in \fkm_L^{i}: v_p(\Delta_{P_{\alpha/F}})=c\}$ to each element of $(i,c,F)$ of $\cP$, we see that the recursive step outlined above is simply the process of moving down on this partial order. Now, the only base cases we know is that we can find the measure of $\{\alpha \in \fkm_L^{i}: v_p(\Delta_{P_{\alpha/F}})=c\}$ when $F=L$, or when $c < 0$ since nothing in $\cO_L$ can have discriminant with negative valuation. Let $BC = \{(i,c,F):F=L \text{ or } c<0\} \subset \cP$ represent these base cases. Now, if $\cP$ were bounded from below by $BC$, then this recursive step would be enough for us to calculate all  $\mu_L(\{\alpha \in \fkm_L^{i}: v_p(\Delta_{P_{\alpha/K}})=c\})$ from our base cases, but it's not. Indeed, as we saw from the above paragraph, in the first case where $a_1=0$, we merely reduced calculating $\mu_L(\{\alpha \in \fkm_L: v_p(\Delta_{P_{\alpha/\bbQ_p}})=c\})$ to calculating $\mu_L(\{\alpha \in \fkm_L^2: v_p(\Delta_{P_{\alpha/\bbQ_p}})=c\})$. Actually, our recursive step will just keep pushing this along, and move from $(i,c,\bbQ_p)$ to $(i+1,c,\bbQ_p)$ to $(i+2,c,\bbQ_p)$, and so on. So, we need a further input to close off the recursion. The trick here is simply to notice that \begin{equation}\label{eqn:temp131313222}
    \mu_L(\{\alpha \in \fkm_L^2: v_p(\Delta_{P_{\alpha/\bbQ_p}})=c\})=p^{-6}\mu_L(\{\alpha \in \fkm_L^2: v_p(\Delta_{P_{\alpha/\bbQ_p}})=c-30\})
\end{equation} where we have used the fact that dividing by $p$ changes the $\mu_L$ measure of a set by a factor of $p^{[L:\bbQ_p]}$, and also $$\prod_{\alpha',\alpha'' \text{ conjugates of }\alpha} (\alpha'/p-\alpha''/p) = p^{-30}\prod_{\alpha',\alpha'' \text{ conjugates of }\alpha} (\alpha'-\alpha'').$$ In terms of $\cP$, we are essentially establishing an equivalence relation $(i+e(L/F),c,F) \sim (i,c-\deg(L/F)(\deg(L/F)-1)i/e,\bbQ_p)$. Now the set $\cP/\sim$ carries the induced partial order, which is bounded below by base cases $BC/\sim$.


In summary, calculating $\{\alpha \in \fkm_L^i: v_p(\Delta_{P_{\alpha/K}})=c\}$ is done with the following four steps. \begin{enumerate}
    \item Take a general $\alpha \in \fkm_L^i$ and consider the lowest term $a_ip^{i/e(L)}$ in the Teichm\"uller expansion.
    \item For each choice of $a_i$ in the Teichm\"uller expansion, we get a partition of the Galois conjugates of $\alpha$, with each set in the partition containing Galois conjugates with the same leading term.
    \item Using the partition, reduce the calculation of terms like $\mu(\{\alpha \in \fkm_L^i: v_p(\Delta_{P_{\alpha/K}})=c\})$ to the calculation of terms like $\mu(\{\alpha \in \fkm_L^{i+1}: v_p(\Delta_{P_{\alpha/F}})=c'\})$
    \item When $i$ gets large use the trick similar to Equation \ref{eqn:temp131313222} to reduce the calculation $\{\alpha \in \fkm_L^i: v_p(\Delta_{P_{\alpha/K}})=c\}$ to $\{\alpha \in \fkm_L^{i-e}: v_p(\Delta_{P_{\alpha/K}})=c'\}$ where $c' < c$ when $i$ gets large.
    \item Repeat the above steps until we reach a base case.
\end{enumerate}

The main idea and novelty of this paper is succinctly summarized above and there are little additional conceptual difficulties. Much of the paper is dedicated to the many technical and combinatorial details in these calculations, including the following. Most of the notational baggage comes from these technicalities. \begin{itemize}
    \item We need to enumerate all extensions of a given splitting type. Fortunately, this is not terrible for tame extensions, but it still takes some bookkeeping.
    \item There are extensions like $L=\bbQ_p[\zeta_{p^3-1},(\zeta_{p^3-1}p)^{1/4}]$. So, Teichm\"uller expansions of elements in $\cO_L$ are of the form $\sum a_i (\zeta_{p^3-1}p)^{1/4})^i \in \cO_L$. Now although we can still express the Galois conjugates of $\sum a_i (\zeta_{p^3-1}p)^{1/4})^i \in \cO_L$, and find its Galois conjugates fairly easily, there are combinatorics involved in determining the partitions of Galois conjugates arising from a given $a_i$, and determining the number of $a_i$ giving a specified partition.
    \item $L$ can be a product of field extensions rather than a single field extension.
    \item We need to sum over \textit{all} extensions $L$ with a given splitting type. We only considered one $L$ above.
\end{itemize}

We remark that without summing over all extensions $L$ with a given splitting type, we get only get ``quasi-rationality" for $\rho$, where quasi-rationality means that as a function of $p$, it is rational in arithmetic progressions with a fixed modulus. Although we do not calculate it out in this paper, following through with the same methods outlined above, we would see that the proportion of polynomials in $\bbZ_p[x]$ of degree $9$ generating $\bbQ_p[\zeta_{p^3-1},p^{1/3}]$ is a quasi-rational function in terms of $p$, so that its restriction to $p \equiv 0 \mod 3$ is a rational function, and its restriction to $p \equiv 1,2 \mod 3$ is another rational function.

\subsection{Previous work}

\subsubsection*{Densities of random polynomial with fixed splitting data} The study of statistics of roots of random polynomials dates back to the work of Bloch and P\"olya who proved asymptotic bounds on the number of expected roots of polynomials with coefficients in $\{-1,0,1\}$. Then, Kac gives an exact formula for the expected number of roots of polynomials with Gaussian distributed coefficients. For a survey and related results, there is \cite{dembo2002random} and \cite{nguyen2021random}.

The questions surrounding the expected number of roots of random polynomials over $p$-adic fields were studied in \cite{evans2006expected},\cite{buhler2006probability}, and \cite{kulkarni2021p}. Following, there is the paper \cite{bhargava_cremona_fisher_gajović_2022} (the  conjecture in which is precisely the one we attack in this paper), which shows that the density of $p$-adic polynomials with fixed number of roots is a rational function in $p$, and satisfies the functional equation sending $p$ to $p^{-1}$. There is also the aforementioned paper \cite{caruso2022zeroes}, from which we obtain a crucial formula for finding the average number of roots of a random polynomial in $\cO_L$ for a fixed extension $L$.

\subsubsection*{Carousels}

\begin{figure} 
    \centering
    \includegraphics[scale=.45]{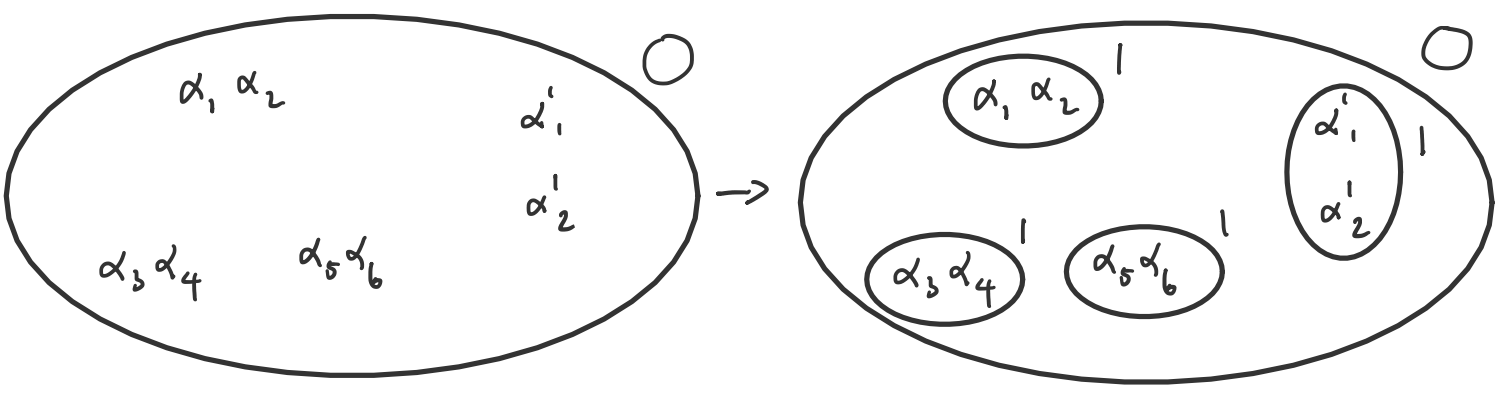}
    \caption{This is a visualization of the main idea behind Proposition \ref{prop:induction_a_with_P}. It is a calculation of the discriminant of the minimal polynomial over $\bbQ_p$ of $(\alpha_1,\alpha_1')=(\zeta_{p^3-1}+p^{1/2},p^{1/2})$. Here $\alpha_2=\zeta_{p^3-1}-p^{1/2}$, $\alpha_3=\zeta_{p^3-1}^p+p^{1/2}$, $\alpha_4=\zeta_{p^3-1}^p-p^{1/2}$, $\alpha_5=\zeta_{p^3-1}^{p^2}+p^{1/2}$, $\alpha_6=\zeta_{p^3-1}^{p^2}-p^{1/2}$, $\alpha_2'=-p^{1/2}$. The discriminant of $P_{(\alpha_1,\alpha_1')/\bbQ_p}$ is the product of the pairwise differences of $\alpha_i$ and $\alpha_i'$. However, we can see clearly from the first terms of the Teichm\"uller expansion of $\alpha_i$ and $\alpha_i'$ the valuation of some pairwise differences. This is visualized in the diagram in that for any two elements not in the same inner circle, such as $\alpha_1$ and $\alpha_1'$, we know that $v_p(\alpha_1-\alpha_1')=0$. Thus, the calculation of the valuation of the discriminant of $P_{(\alpha_1,\alpha_1')/\bbQ_p}$ is reduced to the calculation of valuation of the pairwise products of elements in the inner circles which is precisely the valuation of $P_{\alpha_1/\bbQ_p[\zeta_{p^3-1}]}$ and $P_{\alpha_2/\bbQ_p[0]}$.}\label{fig:recursion}
\end{figure}

In this paper, we make heavy use of the behavior of sets of Galois conjugates under the Galois group. In particular, we consider the distances between Galois conjugates. One may picture this as a carousel, as in \ref{fig:recursion}. In the figure, one sees that the conjugates $\alpha_1$ and $\alpha_2$ are closer to each other than $\alpha_1$ and any other conjugate. As a result, there is a subgroup of the Galois group stabilizing the set $\{\alpha_1,\alpha_2\}$, so pictorially, $\alpha_1,\alpha_2$ are in a carousel within a bigger carousel.

Other papers have made use of such structures. For example, in \cite{dokchitser_dokchitser_maistret_morgan_2022} and \cite{dokchitser_maistret_2020}, one sees that much information about the hyperelliptic curve $y^2=f(x)$ can be obtained from the carousel of roots of $f(x)$. Note that in the above paper, carousels are referred to as ``cluster pictures." 
\section{Notation/Convention}

We adopt a similar notation as in \cite{caruso2022zeroes}. We fix a local field $K$ and an algebraic closure $\bar K$. We equip $\bar K$ with the $p$-adic norm normalized by $||p||=p^{-[K:\bbQ_p]}$. Let $\cO_K$ be the discrete valuation ring of $K$. Let $\fkm_K$ be the maximal ideal of $\cO_K$. Let $\lambda_K$ be the associated Haar measure on $K$. Let $e_{\base}$ be the ramification index of $K$, and $f_{\base}$ be the inertia degree of $K$. Similarly, for an \'etale extension $L = \prod_i L_i$, where $L_i$ are field extensions of $K$, we denote $\cO_L$ to be $\prod_i \cO_{L_i}$ and $\fkm_L$ to be $\prod_i \fkm_{L_i}$.

Let $\Omega_d$ be the space of polynomials of degree at most $d$ with coefficients in $\cO_K$, which is set-theoretically identified with $\cO_K^{d+1}$. With this identification, we have a probability measure $\mu_r$ on $\Omega_r$ associated to $\lambda_K^{\otimes d+1}$ on $\cO_K^{d+1}$. When we refer to a random degree $d$ polynomial, we mean it in this sense. Also, for an \'etale extension $E/K$, and a polynomial $f \in K[x]$, we call the finite set $\Hom(K[x]/f, E)$ the roots of $f$ in $E$. Furthermore, if $\varphi \in \Hom(K[x]/f,E)$ is surjective, we say $\varphi$ is ``new" in $E$. Note that this definition makes sense even for $f$ such that $\deg(f) > \dim_K E$, but we will only concern ourselves with the case where $\deg(f) = \dim_K E$.

Fix a positive integer $d$, and integers $e_1, \dots, e_m$ and $f_1, \dots, f_m$ such that $d=\sum_i e_i f_i$. Let $\sigma$ be the multiset of pairs $\{(e_1,f_1), \dots, (e_m,f_m)\}$, which we will write as $(e_1^{f_1}, \dots, e_m^{f_m})$. We call such symbols splitting types. Let $L$ be an \'etale extension of $K$. We say $\pi_K$ has splitting type $\sigma$ in $L$ if $\pi_K$ factors as $\prod_i \pi_i^{e_i}$ in $\cO_L$, where $\pi_i$ are prime ideals of $\cO_L$ having residue degree $f_i$. For any field $L$, we define $e(L)$ to be the ramification index of $L$ over $\bbQ_p$, where $p$ is the characteristic of the residue field of $L$, and $f(L)$ to be the inertia degree of $L$. Similarly, we define $e(L/K)=e(L)/e(K)$ for any extension $L/K$ to be the relative ramification index of $L$ over $K$, and similarly for $f(L/K)$. Finally, let $[m]$ denote the set $\{1, \dots, m\}$.

Let $q$ be the size of the residue field of $K$. Fix a field extension $L/K$ of inertia degree $f$. Let $\pi_L$ be a uniformizer of $L$. Then, we may write elements of $L$ uniquely as $\sum_{i=-\infty}^{\infty} a_i \pi_L^i$, where $a_i \in L$ is either a $q^f-1$-th root of unity or $0$, and there is some cutoff $r \in \bbZ$ so that $a_i=0$ for $i < r$. This representation of an element in $L$ is called a Teichm\"uller expansion.


\section{Main Results} \label{sec:main}
For a splitting type $\sigma = (e_1^{f_1}, \dots, e_m^{f_m})$, let $\rho(\sigma,K)$ denote the density of polynomials $P \in \cO_K[x]$ of degree $d=\sum_{i=1}^m e_if_i$ such that $P$ generates an \'etale extension $L$ over $K$ and $\pi_K$ has splitting type $\sigma$ in $L$. Let $\alpha(\sigma,K)$ (resp. $\beta(\sigma,K)$) be the density of such polynomials above with the added condition that $P$ is monic (resp. $P$ is monic and $P(x) \mod \pi_K = x^n$. Our main result is the following.
\begin{theorem}\label{thm:main}
Fix a splitting type $\sigma = (e_1^{f_1}, \dots, e_m^{f_m})$.
\begin{enumerate}

    \item If $K,K'$ are two local fields with residue characteristic $p$ such that $\gcd(p,e_i)=1$ for all $i \in [m]$, such that $e(K)=e(K')$ and $f(K)=f(K')$, then $\rho(\sigma,K)=\rho(\sigma,K')$. In particular, we may replace the notation $\rho(\sigma,K)$ with $\rho(\sigma,e,f;p)$ when $\gcd(p,e_i)=1$ for all $i \in [m]$. Similarly, $\alpha(\sigma,K)$ and $\beta(\sigma,K)$ may be replaced with $\alpha(\sigma,e,f;p)$ and $\beta(\sigma,e,f;p)$.
    
    \item Fix $e$ and $f$ as well. Then, for primes $p$ such that $\gcd(p,e_i)=1$ for all $i \in [m]$, $\rho(\sigma,e,f;p)$, $\alpha(\sigma,e,f;p)$, and $\beta(\sigma,e,f;p)$ are rational functions of $q=p^f$.
\end{enumerate}
\end{theorem} We note that it is conjectured in \cite[Conjecture 1.2]{bhargava_cremona_fisher_gajović_2022} that part 2 of this theorem also holds for $p$ not satisfying the gcd condition, but we will not prove that in this paper.

We also have an asymptotic for the large $q$ limit of $\rho$. \begin{theorem} \label{thm:large_q_limit_rho}
Let $q=p^f$. We have $$\rho(\sigma,e,f;p)\sim \frac{1}{\perm(\sigma)\prod_{i=1}^m f_iq^{\sum_{i=1}^m (e_i-1)f_i}}$$ as $q \to \infty$.
\end{theorem}
More precisely, if one takes $p$ to infinity, $f$ to infinity, or both, $$\lim_{p^f \to \infty}\rho(\sigma,e,f;p)\perm(\sigma)\prod_{i=1}^m f_iq^{\sum_{i=1}^m (e_i-1)f_i}$$ approaches $1$. Indeed, this agrees with the calculations \ref{bhargava_computations} done in \cite{bhargava_cremona_fisher_gajović_2022}.

We now launch into the preliminaries of the proof of these theorems. The goal for the rest of this section to transform the problem of finding $\rho(\sigma,K)$ to the problem of finding measures of sets of the form $\mu_L(\{\alpha \in \fkm_L^i: v_p(\Delta_{P_{\alpha/\bbQ_p}})=c\})$. Let $L/K$ be an \'etale extension of degree $d$. For $U \subset L$ an open subset, let $Z_{U,L}^{\new}: \Omega_d \to \bbZ$ be a random variable sending $P$ to the number of new roots of $P$ in $U$. With this notation, we have that 

\begin{lem}
\label{lem:temp1}
Fix a local field $K$ and a splitting type $\sigma=(e_1^{f_1}, \dots, e_m^{f_m})$. For $L=\prod_{i=1}^m L_i$, and $I \subset [m]$ let $U_{L,I} = \prod_{i \in I} \fkm_{L_i} \times \prod_{i \in [m] \setminus I} \cO_{L_i}$. Then, \begin{enumerate}
    \item $$\rho(\sigma,K)=\sum_{\substack{\textup{isom classes of }L, \\ \pi_K \textup{ has splitting type } \sigma \text{ in }L}}\frac{\bbE[Z_{L,L}^{\new}]}{\Aut(L)}.$$
    \item $$\alpha(\sigma,K)=\sum_{\substack{\textup{isom classes of }L, \\ \pi_K \textup{ has splitting type } \sigma \text{ in }L}}\frac{\bbE[Z_{\cO_L,L}^{\new}]}{\Aut(L)}.$$
    \item $$\beta(\sigma,K)=\sum_{I \subsetneq [m]} (-1)^{|I|+1} \sum_{\substack{\textup{isom classes of }L, \\ \pi_K \textup{ has splitting type } \sigma \text{ in }L}}\frac{\bbE[Z_{U_{L,I},L}^{\new}]}{\Aut(L)}.$$
    
\end{enumerate}
\end{lem}
\begin{proof}

We first partition degree $d$ polynomials $P$ by the isomorphism classes of \'etale extensions that they generate. Thus $$\rho(\sigma,K)=\sum_{\substack{\text{isom classes of }L, \\ \pi_K \text{ has splitting type } \sigma \text{ in }L}}\rho(L,K),$$ where $\rho(L,K)$ is the proportion of degree $d$ polynomials $P$ over $K$ such that $K[x]/P \cong L$. In other words, this is the proportion of polynomials with a new root in $L$. Now, for any such $P$, it has $\Aut(L)$ roots in $L$. Thus, $\Aut(L)\rho(L,K)$ is precisely the average number of new roots of a random degree $d$ polynomial in $L$. Hence, $$\Aut(L)\rho(L,K)=\bbE[Z_{L,L}^{\new}].$$ The formula for $\alpha$ follow after noting that monic polynomials are precisely the ones for which all roots lie in the ring of integers of the associated \'etale extension. Now, monic polynomials $P(x)$ that is $x^n \mod \pi_K$ are precisely the ones for which all roots lie in $\cO_L$, with at least one root lies in $\fkm_L$. Thus, we obtain the formula for $\beta$ after applying inclusion exclusion.
\end{proof}

In the following theorem \cite[Theorem 4.8]{caruso2022zeroes}, Caruso gives a formula for $\bbE[Z_{U,L}]$, for open subsets $U \subset \cO_L$. Let $w(d)=w_{f(K)}(d)=\frac{q^{d+1}-q^d}{q^{d+1}-1}$.
\begin{theorem}[Theorem 4.5.5 \cite{caruso2022zeroes}]
\label{thm:Xavier_Integral} 
Let $K$ be a local field, and $L$ be an \'etale extension of degree $d$. Let $U \subset \cO_L$ be any open set. Then, the average number of new roots of a random degree $d$ polynomial over $K$ in $L$ is $$\bbE[Z_{U,L}^{\new}]=||D_L||w(d)\int_{\alpha \in U} \frac{1}{[\cO_L:\cO_K[\alpha]]} d\lambda_L.$$
\end{theorem}

Thus, we have turned a statement about polynomials over $K$ to a statement about elements in \'etale extensions of $K$. We can apply the above formula directly to calculate $\bbE[Z_{\cO_L,L}^{\new}]$ and $\bbE[Z_{U_L,L}^{\new}]$ and thus obtain $\alpha$ and $\beta$. However, since we require $U$ to be a subset of $\cO_L$ in the above formula, we see that we cannot use this to calculate $\bbE[Z_{L,L}]$ and obtain $\rho$ directly. However, we will prove the following in Section \ref{sec:red_to_OL_thm_proof}.
\begin{theorem} 
\label{thm:reduce_to_OL}
Let $K$ be a local field and let $q=p^{f(K)}$. Let $L$ be an \'etale extension of degree $d$. So, $L=\prod_{i=1}^m L_i$, where $L_i$ are field extensions of $K$. For any $A \subset [m]$, denote $L_A = \prod_{i \in A} L_i$ and $A^c = [m] \setminus A$. Then, $$\bbE[Z_{L,L}^{\new}]=\sum_{A \subset [m]}\frac{w([L:K])}{w([L_A:K])w([L_{A^c}:K])} \bbE[Z_{\cO_{L_A},L_A}^{\new}]\bbE[Z_{\fkm_{L_{A^c}},L_{A^c}}^{\new}],$$ where we make the convention that if $A$ is empty, $\bbE[Z_{\cO_{L_A},L_A}^{\new}]=\bbE[Z_{\fkm_{L_A},L_A}^{\new}]=1$ and $[L_A:K]=0$ (so that $w([L_A:K])=1$).
\end{theorem}

We are concerned mainly with tamely ramified extensions. Thus, with Theorem \ref{thm:reduce_to_OL} and the following characterization of tamely ramified extensions, we can transform our sum $$\sum_{\substack{\text{isom classes of }L, \\ \pi_K \text{ has splitting type } \sigma \text{ in }L}}\frac{\bbE[Z_{L,L}^{\new}]}{\Aut(L)},$$ which involves terms of the form $\bbE[Z_{L,L}^{\new}]$, to something that involves only terms of the form $\bbE[Z_{\cO_L,L}]$ and $\bbE[Z_{\fkm_L,L}]$, which we should be able to calculate using Caruso's formula (\ref{thm:Xavier_Integral}). The theorem below is given in \cite[Chapter 16]{hasse_zimmer_1980}.
\begin{theorem}
\label{thm:tamely_ram_extns}
Fix a uniformizer $\pi_K \in K$ and a positive integer $n$ coprime to $p$. Then, there are $n$ tamely and totally ramified extensions of degree $n$ of $K$ in $\bar K$. These extensions may be split into $\gcd(q-1,n)$ isomorphism classes, each class isomorphic to $K[(\zeta_{q-1}^j \pi_K)^{1/n}]$, for some $0 \leq j < \gcd(q-1,n)$. 
\end{theorem} 
Since every tamely ramified extension is an unramified extension followed by a tamely and totally ramified extension, we see that every tamely ramified extension $L$ with ramification index $e_0$ and inertia degree $f_0$ over $K$ is of the form $K[\zeta_{q^{f_0}-1}, (\zeta_{q^{f_0}-1}^j \pi_K)^{1/e_0}]$ for some $0 \leq j < \gcd(q^{f_0)}-1,e_0)$. We will now denote this field as $L_j$. Thus, extensions $L/K$ with splitting type $\sigma$ are of the form $\prod_i L_{j_i}$, where $0 \leq j_i \leq \gcd(q^{f_i}-1,e_i)$. This also shows that $|\Aut(L_{j_i}/K)|=f_i\gcd(q^{f_i}-1,e_i)$.

Let $$\perm(\sigma)=\prod_{e,f} (\text{occurrences of } e^f \text{ in }\sigma)!.$$ Now we are ready to transform our formula in terms of $\bbE[Z_{L,L}]$, to a formula in terms $\bbE[Z_{\cO_L,L}]$ and $\bbE[Z_{\fkm_L,L}]$. With Lemma \ref{lem:temp1}, Theorem \ref{thm:reduce_to_OL}, and Theorem \ref{thm:tamely_ram_extns}, we have

\begin{cor}
\label{cor:temp111}
Let $\sigma=(e_1^{f_1}, \dots, e_m^{f_m})$ be a splitting type with $p \nmid e_i$. Let $d=\sum e_if_i$. Then, $$\rho(\sigma,K)=\frac{w(d)\sum_{A \subset [m]} N_{1,A}'N_{2,A}'}{\perm(\sigma)}$$ where $$N_{1,A}'= \frac{\sum_{\substack{L_A = \prod_{i \in A} L_{j_i}\\ j_i \in [\gcd(q^{f_i}-1,e_i)]}}\bbE[Z_{\cO_{L_A},L_A}^{\new}]}{\prod_{i \in A}f_i\gcd(q^{f_i}-1,e_i)w([L_A:K])}$$ and $$N_{2,A}'=\frac{\sum_{\substack{L_{A^c} = \prod_{i \notin A} L_{j_i}\\ j_i \in [\gcd(q^{f_i}-1,e_i)]}}\bbE[Z_{\fkm_{L_{A^c}},L_{A^c}}^{\new}]}{\prod_{i \notin A} f_i\gcd(q^{f_i}-1,e_i)w([L_{A^c}:K])}.$$
\end{cor}

\begin{proof}
We first have by Lemma \ref{lem:temp1} that $$\rho(\sigma,K)=\sum_{\substack{\text{isom classes of }L, \\ \pi_K \text{ has splitting type } \sigma \text{ in }L}}\frac{\bbE[Z_{L,L}^{\new}]}{\Aut(L)}.$$ Now as $j_i$ ranges in $[\gcd(q^{f_i}-1),e_i)]$, $L=\prod_i L_{j_i}$ ranges among \'etale extensions of splitting type $\sigma$ over $K$. However, this will hit the same extension $\perm(\sigma) \left| \frac{\Aut(L)}{\prod_i \Aut(L_{j_i})}\right|^{-1}$ times. Thus, \begin{align*}
    \sum_{\substack{\text{isom classes of }L, \\ \pi_K \text{ has splitting type } \sigma \text{ in }L}}\frac{\bbE[Z_{L,L}^{\new}]}{\Aut(L)}&=\perm(\sigma)^{-1}\prod_i |\Aut(L_{j_i})|^{-1}\sum_{\substack{L=\prod_i L_{j_i}\\ j_i \in [\gcd(q^{f_i}-1,e_i)]}}\bbE[Z_{L,L}^{\new}]\\
    &=\perm(\sigma)^{-1}\prod_i \frac{1}{f_i\gcd(q^{f_i}-1,e_i)}\sum_{\substack{L=\prod_i L_{j_i}\\ j_i \in [\gcd(q^{f_i}-1,e_i)]}}\bbE[Z_{L,L}^{\new}]
\end{align*}

Now, by Theorem \ref{thm:reduce_to_OL}, \begin{align*}
    \sum_{\substack{L=\prod_i L_{j_i}\\ j_i \in [\gcd(q^{f_i}-1,e_i)]}}\bbE[Z_{L,L}^{\new}]&=\sum_{\substack{L=\prod_i L_{j_i}\\ j_i \in [\gcd(q^{f_i}-1,e_i)]}}\sum_{A \subset [m]} \frac{w(d)}{w([L_A:K])w([L_{A^c}:K])}\bbE[Z_{\cO_{L_A},L_A}^{\new}]\bbE[Z_{\fkm_{L_{A^c}},L_{A^c}}^{\new}]\\
    &=w(d)\sum_{A \subset [m]}\sum_{\substack{L=\prod_i L_{j_i}\\ j_i \in [\gcd(q^{f_i}-1,e_i)]}} \frac{\bbE[Z_{\cO_{L_A},L_A}^{\new}]}{w([L_A:K])}\frac{\bbE[Z_{\fkm_{L_{A^c}},L_{A^c}}^{\new}]}{w([L_{A^c}:K])}
\end{align*}
where $L_A=\prod_{i \in A} L_{j_i}$ and $L_{A^c} = \prod_{i \in A^c} L_{j_i}$. Thus, factoring the inner sum into a product, we get $$\sum_{\substack{L=\prod_i L_{j_i}\\ j_i \in [\gcd(q^{f_i}-1,e_i)]}}\bbE[Z_{L,L}^{\new}]=w(d)\sum_{A \subset [m]} \left(\sum_{\substack{L_A = \prod_{i \in A} L_{j_i}\\ j_i \in [\gcd(q^{f_i}-1,e_i)]}}\frac{\bbE[Z_{\cO_{L_A},L_A}^{\new}]}{w([L_A:K])}\right)\left(\sum_{\substack{L_{A^c} = \prod_{i \in A^c} L_{j_i}\\ j_i \in [\gcd(q^{f_i}-1,e_i)]}}\frac{\bbE[Z_{\fkm_{L_{A^c}},L_{A^c}}^{\new}]}{w([L_{A^c}:K])}\right).$$
\end{proof}

Although the formula in Corollary \ref{cor:temp111} seems complicated, one just needs to note that since by definition $w(d)$ is a rational function of $q=p^{f(K)}$, Theorem \ref{thm:main} for $\rho$ will be proven if we can show that $$\frac{1}{\prod_{i \in A} f_i\gcd(q^{f_i}-1,e_i)}\sum_{\substack{L_A = \prod_{i \in A} L_{j_i}\\ j_i \in [\gcd(q^{f_i}-1,e_i)]}}\bbE[Z_{\cO_{L_A},L_A}^{\new}]$$ and $$\frac{1}{\prod_{i \in A^c} f_i\gcd(q^{f_i}-1,e_i)}\sum_{\substack{L_{A^c} = \prod_{i \in A^c} L_{j_i}\\ j_i \in [\gcd(q^{f_i}-1,e_i)]}}\bbE[Z_{\fkm_{L_{A^c}},L_{A^c}}^{\new}]$$ are rational functions in terms of $q$, and that these expressions depends on $K$ through $e(K)$ and $f(K)$. Similarly, following the same process for $\alpha$ and $\beta$, one simply needs to show that $$\prod_i \frac{1}{f_i\gcd(q^{f_i}-1,e_i)}\sum_{\substack{L=\prod_i L_{j_i}\\ j_i \in [\gcd(q^{f_i}-1,e_i)]}}\bbE[Z_{\cO_L,L}^{\new}]$$ and $$\prod_i \frac{1}{f_i\gcd(q^{f_i}-1,e_i)}\sum_{\substack{L=\prod_i L_{j_i}\\ j_i \in [\gcd(q^{f_i}-1,e_i)]}}\bbE[Z_{U_{L,I},L}^{\new}]$$ ($U_{L,I}$ defined as in Lemma \ref{lem:temp1}) are rational functions in terms of $q$, and that these expressions depends on $K$ through $e(K)$ and $f(K)$. In fact, we will show something more general, that this is true for $$\frac{1}{\prod_{i=1}^m f_i\gcd(q^{f_i}-1,e_i)}\sum_{\substack{j_i \in [\gcd(q^{f_i}-1,e_i)]}}\bbE[Z_{\prod_{i=1}^m \fkm_{L_i}^{b_i},\prod_i L_{j_i}}^{\new}]$$ for any tamely ramified splitting type $\sigma=(e_1^{f_1}, \dots, e_m^{f_m})$ and nonnegative integers $b_1, \dots, b_m$.


Note that combining Theorem \ref{thm:Xavier_Integral} with Corollary \ref{cor:temp111} gives us the formula
\begin{equation} \label{eqn:rho}
    \rho(\sigma,K)=\frac{w(d)\sum_{A \subset [m]}N_{1,A}N_{2,A}}{\perm(\sigma)q^{\sum_{i=1}^m (e_i-1)f_i}\prod_{i=1}^m f_i},
\end{equation} where $$N_{1,A}=\frac{\sum_{\substack{L_A = \prod_{i \in A} L_{j_i}\\ j_i \in [\gcd(q^{f_i}-1,e_i)]}}\int_{\alpha \in \cO_{L_A}} \frac{1}{[\cO_{L_A}:\cO_K[\alpha]]} d\lambda_{L_A}}{\prod_{i \in A}\gcd(q^{f_i}-1,e_i)}$$ and $$N_{2,A}=\frac{\sum_{\substack{L_{A^c} = \prod_{i \notin A} L_{j_i}\\ j_i \in [\gcd(q^{f_i}-1,e_i)]}}\int_{\alpha \in m_{L_{A^c}}} \frac{1}{[\cO_{L_{A^c}}:\cO_K[\alpha]]} d\lambda_{L_{A^c}}}{\prod_{i \notin A} \gcd(q^{f_i}-1,e_i)}.$$

Finally, the integral in Caruso's formula (\ref{thm:Xavier_Integral}) involves terms of the form $[\cO_L:\cO_K[\alpha]]$, which we will transform with the following lemma.
\begin{lemma}
\label{lem:cond_disc_order}
Let $P$ be a monic polynomial in $\cO_K[x]$, which generates an \'etale extension $L=\prod_{i}L_i/K$. Let $\Delta_P$ be the polynomial discriminant of $P$, and let $\Delta_L$ be the discriminant of the \'etale extension $L$ (aka $N(\prod_i \pi_{L_{j_i}}^{e(L_{j_i})-1})$, where $\pi_L \in \cO_L$ is a uniformizer). Then, $\pi_K^{2\log_q[\cO_L:\cO_K[x]/P]}\Delta_L=\Delta_P$. 
\end{lemma}

Thus, combining Theorem \ref{thm:Xavier_Integral} and Lemma \ref{lem:cond_disc_order}, we can reduce our problem to showing that $$\frac{q^{\frac{\sum(e_i-1)f_i}{2}}}{\prod_{i=1}^m \gcd(q^{f_i}-1,e_i)}\sum_{\substack{L=\prod_i L_{j_i}\\ j_i \in [\gcd(q^{f_i}-1,e_i)]}}\int_{\alpha \in \prod_i \fkm_{L_{j_i}}^{b_i}} q^{-\frac{1}{2}v_K(\Delta_{P_{\alpha/K}})} d\lambda_L$$ is a rational function in terms of $q$ and depends on $K$ only through $e(K)$ and $f(K)$. Sections \ref{sec:generating_function} and \ref{sec:proof_of_main_result} are dedicated to the calculation of this sum and showing that it is rational function in terms of $q$.

\section{Generating Functions and Recurrence Relations} \label{sec:generating_function}

First, we will make a switch from relative notation to absolute notation. So, let $\{L_i\}_{i=1}^m$ be field extensions of $K$. We define $e_i,f_i$ to be the \textit{absolute} ramification index and inertia degree of $L_i$. Now, for nonnegative integers $b_i$, we define $$a(\{L_i\}_{i=1}^m,K,c,\{b_i\}_{i=1}^m):=\mu(\{\alpha \in \prod_{i \in [m]} \fkm_{L_i}^{b_i}: v_p(\Delta_{P_{\alpha/K}})=c\}).$$ Also, fix two positive integers $e_{\base}$, $f_{\base}$, representing the ramification index and inertia degree of our base extension, which we will denote by $K$. Then, let $\sigma=(e_1^{f_1}, \dots, e_m^{f_m})$ be a tame splitting type, so that $p \nmid e_i/e_{\base}$ for any $i$. Let $e_{i,\rel}=e_i/e_{\base}$ and $f_{i,\rel}=f_i/f_{\base}$. Then, we define $$a(\sigma,e_{\base},f_{\base},c,\{b_i\}_{i=1}^m,p) = \frac{1}{\prod_{i \in [m]} \gcd(p^{f_i}-1,e_{i,\rel})}\sum_{j_i \in [\gcd(p^{f_i}-1,e_{i,\rel})]} a(\{L_{j_i}\},K,c,\{b_i\}_{i=1}^m)$$ where $K$ is any $p$-adic field with $e(K)=e_{\base}$ and $f(K)=f_{\base}$. As the notation suggests, the value of this expression will be shown to only depend on $K$ through $e(K)$ and $f(K)$. We will often omit the dependence of $a$ on $p$.

\begin{example} \label{ex:base_case_a}
As a base case, we can see that from the definition, $$a((e^f),e,f,c,\{b\})=a(\{K\},K,c,\{b\})=\begin{cases}
p^{-bf} & c = 0 \\
0 & c \neq 0
\end{cases},$$ for any $j$ since the discriminants of all degree $1$ polynomials are all $1=p^0$.
\end{example}

Our main proposition of this section, as well as the main technical part of our paper is
\begin{prop}
\label{prop:induction_a_with_P}
Fix a tamely ramified splitting type $\sigma=(e_1^{f_1}, \dots, e_m^{f_m})$. Then, the definition of $$a(\sigma,e_{\base},f_{\base},c,\{b_i\})$$ doesn't depend on the choice of $K$ (as long as $e(K)=e_{\base}$ and $f(K)=f_{\base}$). Moreover, we have the following recurrence.

$$a(\sigma,e_{\base},f_{\base},c,\{b_i\})=\sum_{\substack{\textup{partitions } \{E_l\}_{l=1}^k \text{ of }[m] \\ \{n_{O,l}\}_{l=1}^k \in \bbZ_{\geq 1}^k}} P(\sigma,\{E_l\},\{b_i\},\{n_{O,l}\},p)$$$$\sum_{\substack{\{c_l\}_{l=1}^k \in \bbZ_{\geq 0}^k \\ \sum n_{O,l} c_l=r(c)}} \prod_l a(\sigma_l, e(K)h(\beta,n_{O,l}), f(K)\frac{n_{O,l}}{h(\beta,n_{O,l})}, c_l, \{b_i+\bb1_{I}(i)\}_{i \in E_{l}}).$$ Where $\bb1_{I}(i)$ denotes the indicator function of $I$ and $\sigma_l=(e_i^{f_i})_{i \in E_l}$.
\end{prop}

The notation used in this proposition will be introduced shortly. The main idea of this proposition is that we are expressing $a(\sigma,e_{\base},f_{\base},c,\{b_i\})$ in terms of $$a(\sigma_l, e(K)h(\beta,n_{O,l}), f(K)\frac{n_{O,l}}{h(\beta,n_{O,l})}, c_l, \{b_i+\bb1_{I}(i)\}_{i \in E_{l}}).$$ After introducing the appropriate notations, it will be apparent that this seemingly complicated term satisfies the four conditions: \begin{enumerate}
    \item $\sigma_l \subset \sigma$,
    \item $h(\beta,n_{O,l}) \geq 1$,
    \item $\frac{n_{O,l}}{h(\beta,n_{O,l})} \geq 1$,
    \item and $\min_i \{b_i+\bb1_{I}(i)\}_{i \in E_{l}} > \min_i \{b_i\}$.
\end{enumerate} By induction on degree, we have theoretically already calculated the term $$a(\sigma_l, e(K)h(\beta,n_{O,l}), f(K)\frac{n_{O,l}}{h(\beta,n_{O,l})}, c_l, \{b_i+\bb1_{I}(i)\}_{i \in E_{l}})$$ whenever one of the first three conditions is satisfied strictly. Thus, this proposition effectively reduces the calculation of $a(\sigma,e_{\base},f_{\base},c,\{b_i\})$ to $a(\sigma,e_{\base},f_{\base},c,\{b_i+\bb1_I(i)\})$. 

Now we introduce the relevant terms in the proposition. Fix a tamely ramified splitting type $(e_1^{f_1}, \dots, e_m^{f_m})$. Let $d_i = e_{i,\rel}f_{i,\rel}$, and so $d=\sum_i d_i$. Also, let $\beta=\beta(\{e_i\},e_{\base},\{b_i\})=\min \frac{b_i}{e_{i,\rel}}$. In other words, for any choices of $i \in [m]$ and $j_i \in [\gcd(p^{f_i}-1,e_{i,\rel})]$, $\beta$ is the smallest possible valuation of any element in $\fkm_{L_{j_i}}^{b_i}$. We will often omit the dependence of $\beta$ on $\{e_i\},e_{\base},$ and $\{b_i\}$, when these inputs are clear from context. Let $I=I(\{e_i\},e_{\base},\{b_i\})=\{i \in [m]: \frac{b_i}{e_{i,\rel}}=\beta$. Let $\denom(\beta)$ denote the denominator of $\beta$ in its lowest form. Note that if $\frac{a}{b}=\beta$, then $\denom(\beta)=\frac{b}{\gcd(a,b)}$. Let $T_{f_i,p}$ denote the Teichm\"uller representatives (aka $p^{f_i}-1$ roots of unity and $0$). 

We make a few more definitions which are more ad hoc, and so the reasons that we have defined these terms as such may be clearer if one reads the proof of Proposition \ref{prop:induction_a_with_P}, and refer to these definitions as needed. Fix a splitting type $\sigma = (e_1^{f_1}, \dots, e_m^{f_m})$, a partition $\{E_l\}_{l=1}^j$ of $[m]$, an integer $n_{O,l}$ associated to each element $E_l$ of the partition, and an integer $b_i$ for each $i \in [m]$. Let $$r(c)=r(c,\{E_l\}, \{n_{O,l}\})=c-\frac{\beta}{e_{\base}}(d(d-1)-\sum_{l} \sum_{i \in E_l}d_i(\sum_{i \in E_l}(d_i/n_{O,l})-1)),$$ and $$h(\beta,n_{O,l}) = \begin{cases} \denom(\beta) & n_{O,l} \neq 1 \\
1 & n_{O,l}=1
\end{cases}.$$ Now consider the set $$ \prod_{i=1}^m [\gcd(p^{f_i}-1,e_{i,\rel}] \times \prod_{i \in I} T_{f_i,p}.$$ We consider three conditions on elements $(\{j_i\}_{i=1}^m, \{a_{i,b_i}\}_{i \in I})$ in this set: 
\begin{enumerate}
    \item For all $i_1,i_2 \in I$, $i_1,i_2$ belong to the same $E_l$ if and only if $a_{i_1,b_{i_1}}\pi_{L_{j_{i_1}}}^{b_{i_1}}$ and $a_{i_2,b_{i_2}}\pi_{L_{j_{i_2}}}^{b_{i_2}}$ are Galois conjugates.
    \item If $i \in I$, then $n_O(a_{i,b_{i}}\pi_{L_{j_{i}}}^{b_{i}})=n_{O,l}$.
    \item If $I \neq [m]$, then there is some $E_{l_0}$ containing all of $[m] \setminus I$. Moreover, for any $i \in E_{l_0} \cap I$, $a_{i,b_i}=0$.
\end{enumerate}
Let $P(\sigma, \{E_l\},\{b_i\},\{n_{O,l}\},p)$ be $\frac{1}{\prod_{i \in [m]}\gcd(p^{f_i}-1,e_{i,\rel})}$ times the number of elements in $\prod_{i=1}^m [\gcd(p^{f_i}-1,e_{i,\rel}] \times \prod_{i \in I} T_{f_i,p}$ satisfying these conditions. Note that if there is such an $l_0$ as in the last condition, then unless $n_{O,l_0}=1$, the second and third condition will contradict each other, and $P(\sigma, \{E_l\},\{b_i\},\{n_{O,l}\},p)$ in this case would just be $0$.

To give a bit more context to these notation, we will relate some of these notations to the overview of proof given at the beginning. 
\begin{center}
\begin{tabular}{||c | c||} 
 \hline
 Notation & Corresponding idea in overview \\ [0.5ex] 
 \hline\hline
 $\beta$ & lowest degree \\
 \hline
 $n_{O,l}$ & number of Galois conjugates of the lowest degree term \\
 \hline
 $\{b_i+\bb1_I(i)\}_{i \in E_l}$ & going from $i$ to $i+1$ \\ 
 \hline
 $P(\sigma, \{E_l\}, \{b_i\}, \{n_{O,l}\},p)$ & \makecell{number of choices of Teichm\"uller representatives $a_i$ \\ giving a certain partition of Galois conjugates} \\ 
 \hline
 $r(c)$ & $c'$ \\
 \hline
\end{tabular}
\end{center}


\begin{proof}[Proof of Proposition \ref{prop:induction_a_with_P}] 
First fix choices $\{j_i\}_{i \in [m]}$, and consider elements of $L_{j_i}$ in power series expansion using Teichm\"uller representatives and the uniformizer $\pi_{L_{j_i}}=(\zeta_{p^{f_{ext}}-1}^{j_i}\pi_K)^{1/e_{i,\rel}}$. Then, the elements in $\fkm_{L_{j_i}}^{b_i}$ are of the form $\alpha_i=\sum_{n=b_i}^{\infty} a_{i,n} \pi_{L_{j_i}}^n$. Our proposition will be based on a rewriting of the following: $$\mu(\{\{\alpha_i\}_{i \in [m]} \in \prod_{i \in [m]} \fkm_{L_{j_i}}^{b_i}: v_p(\Delta_{P_{(\alpha_i)/K}})=c\})$$$$=\sum_{\substack{i \in I \\ a_{i,b_i} \in T_{f_i,p}}}\mu(\{\{\alpha_i\}_{i \in [m]} \in \prod_{i \in [m]} \fkm_{L_{j_i}}^{b_i}: \text{ for all }i \in I, \alpha_i \in a_{i,b_i}\pi_{L_{j_i}}^{b_i}+\alpha_{i'}, \alpha_i' \in \fkm_{L_{j_i}}^{i+1}, v_p(\Delta_{P_{(\alpha_i)}})=c\}.$$ Now, let $(\alpha_1, \dots, \alpha_m) \in \prod_{i \in [m]} L_{j_i}$, then the discriminant of its minimal polynomial $P_{(\alpha_1, \dots, \alpha_m)}$ is equal to $$\prod_{(i,k) \neq (i',k'), k,k' \in [d_i]} (\alpha_{i,k}-\alpha_{i',k'}),$$ where $\alpha_{i,k}$ are Galois conjugates of $\alpha_i$ over $K$. Our first step is to calculate the $p$-adic valuation of minimal polynomials of elements in $$\{\{\alpha_i\}_{i \in [m]} \in \prod_{i \in [m]} \fkm_{L_{j_i}}^{b_i}: \text{ for all }i \in I, \alpha_i \in a_{i,b_i}\pi_{L_{j_i}}^{b_i}+\alpha_{i'}, \alpha_i' \in \fkm_{L_{j_i}}^{i+1}\}$$ in a clever way as in Figure \ref{fig:recursion}. The next few paragraphs will explain in rigorous detail the idea reflected in the figure.

Define $$C=C(\{L_i\},K,\{b_i\},\{a_{i,b_i}\}_{i \in I})$$ to be the set of Galois conjugates of $a_{i,b_i} \pi_{L_{j_i}}^{b_i}$, as $i$ varies over $I$. If there is an $i$ so that $b_i/e_{i,\rel} \neq \beta$, we also add $0$ to $C$. Define $\bar \gamma$ to be the set of Galois conjugates of $\gamma$, and similarly define $\bar C$ to be $C$ up to Galois conjugates. This way, for every $i,k$, with $k \in [d_i]$, $$\alpha_{i,k}=\gamma_{i,k} + H.O.T.,$$ for some $\gamma_{i,k} \in C$. Now we partition the set $\{(i,k):k \in [d_i]\}$ with the equivalence relation $(i,k) \sim (i',k')$ if $\bar \gamma_{i,k} = \bar \gamma_{i',k'}$. Let $S_{\gamma}$ denote these partitions. Then, \begin{equation} \label{eqn:disc_recursion}
    \prod_{(i,k) \neq (i',k'), k,k' \in [d_i]} (\alpha_{i,k}-\alpha_{i',k'})=\prod_{(i,k) \not \sim (i',k'), k,k' \in [d_i]}(\alpha_{i,k}-\alpha_{i',k'})\prod_{\gamma \in C}\prod_{\substack{(i,k),(i',k') \in S_{\gamma} \\ (i,k) \neq (i',k')}}(\alpha_{i,k}-\alpha_{i',k'}).
\end{equation} 

In the case where $(i,k) \not \sim (i',k')$, $\gamma_{i,k} \neq \gamma_{i',k'}$, and so if $v_K(\gamma_{i,k}-\gamma_{i',k
})=\beta$, then since $\alpha_{i,k}=\gamma_{i,k} + H.O.T.,$ we would also have that $v_K(\alpha_{i,k}-\alpha_{i',k'})=\beta$. Now, our $\gamma$'s are either $0$ or of the form $$a_{i,b_i} \pi_{L_{j_i}}^{b_i}=\zeta_{p^{f_i}-1}^{l_i}(\zeta_{p^{f_i}-1}^{j_i}\pi_K)^{b_i/e_{i,\rel}}.$$ Thus, $\gamma_{i,k}-\gamma_{i',k'}$ will be of the form $$\zeta_{e(p^f-1)}^l \pi_K^{\beta}-\zeta_{e'(p^{f'}-1)}^{l'} \pi_K^{\beta}=\zeta_{e(p^f-1)}^l(1-\frac{\zeta_{e'(p^{f'}-1)}^{l'}}{\zeta_{e(p^f-1)}^l})\pi_K^{\beta}$$ and since $p \nmid e$ and $p \nmid e'$, we see that the order of the root of unity $\frac{\zeta_{e'(p^{f'}-1)}^{l'}}{\zeta_{e(p^f-1)}^l}$ is not divisible by $p$, and therefore $$v_p(\zeta_{e(p^f-1)}^l(1-\frac{\zeta_{e'(p^{f'}-1)}^{l'}}{\zeta_{e(p^f-1)}^l})\pi_K^{\beta})=v_K(\pi_K^{\beta}),$$ as desired. Now, the number of choices of pairs $(i,k),(i',k')$ so that $(i,k) \not \sim (i',k')$ is equal to $$d(d-1)-\sum_{\gamma} |S_{\gamma}|(|S_{\gamma}|-1).$$ Thus, \begin{equation}\label{eqn:temp131313}
    v_K(\prod_{(i,k) \not \sim (i',k'), k,k' \in [d_i]}(\alpha_{i,k}-\alpha_{i',k'}))=\frac{1}{e_{\base}}(d(d-1)-\sum_{\gamma} |S_{\gamma}|(|S_{\gamma}|-1))\beta
\end{equation} 

Now we compute the term $$\prod_{\gamma \in C}\prod_{\substack{(i,k),(i',k') \in S_{\gamma} \\ (i,k) \neq (i',k')}}(\alpha_{i,k}-\alpha_{i',k'}).$$ We may express the inner product $$\prod_{\substack{(i,k),(i',k') \in S_{\gamma} \\ (i,k) \neq (i',k')}}(\alpha_{i,k}-\alpha_{i',k'})$$ as discriminant of some the minimal polynomial over $K[\gamma]$ of some element in the \'etale extension of $K[\gamma]$. One can see that since $S_{\gamma}$ consists of all pairs $(i,k)$ such that $\gamma_{i,k}=\gamma$, this element may be constructed as the tuple consisting of $\alpha_{i,k_i}$, as $i$ ranges between elements of $[m]$ such that $\gamma$ is a Galois conjugate of $\gamma_i$, and $k_i$ is a choice of any Galois conjugate of $\alpha_{i,k_i}$ so that $\gamma_{i,k_i}=\gamma$. 

Now let $\bar C = C/\sim$, where $\gamma \sim \gamma'$ if they are Galois conjugates. Then, $\bar C$ gives a partition $\{E_{\bar \gamma}\}$ of $[m]$ via the equivalence $i \sim i'$ if $\bar \gamma_i = \bar \gamma_{i'}$. Thus, $E_{\bar \gamma}$ would consist of all indices $i$ so that $\bar \gamma_i = \bar \gamma$. From now on, we fix representatives $\gamma$ of $\bar \gamma$. With these representatives, for any element $i \in E_{\bar \gamma}$, we may define $\psi_{\gamma_i} \in Gal(K)$ be any element such that $\psi_{\gamma_i}(\gamma_i)=\gamma$. Now we can express what we shown in the previous paragraph as the fact that $$\prod_{\substack{(i,k),(i',k') \in S_{\gamma} \\ (i,k) \neq (i',k')}}(\alpha_{i,k}-\alpha_{i',k'})$$ is equal to the discriminant of the minimal polynomial of $(\psi_{\gamma_i}(\alpha_i)-\gamma)_{i\in E_{\bar \gamma}}$ over $K[\gamma]$. Thus, we have \begin{align*}
    v_p(\prod_{\gamma \in C}\prod_{\substack{(i,k),(i',k') \in S_{\gamma} \\ (i,k) \neq (i',k')}}(\alpha_{i,k}-\alpha_{i',k'}))
    &=v_p(\prod_{\gamma \in C}\Delta_{P_{(\psi_{\gamma_i}(\alpha_i)-\gamma)}/K[\gamma]})\\
    &=v_p(\prod_{\bar \gamma \in \bar C} \prod_{\gamma \in \bar \gamma}\Delta_{P_{(\psi_{\gamma_i}(\alpha_i)-\gamma)}/K[\gamma]})\\
    &=\sum_{\bar \gamma \in \bar C} \# \bar \gamma v_p(\Delta_{P_{(\psi_{\gamma_i}(\alpha_i)-\gamma)}/K[\gamma]}),
\end{align*} where in RHS of the last equality, $\gamma$ is any chosen representative of $\bar \gamma$.



Now it is clear from the definition that $|S_{\gamma}|=\sum_{i \in E_{\bar \gamma}}d_i/\# \bar \gamma $. Thus, $$\frac{\beta}{e_{\base}}(d(d-1)-\sum_{\bar \gamma} |S_{\gamma}|(|S_{\gamma}|-1)) = \frac{\beta}{e_{\base}}(d(d-1)-\sum_{\bar \gamma \in \bar C} \sum_{i \in E_{\bar \gamma}}d_i(\sum_{i \in E_{\bar \gamma}}(d_i/\# \bar \gamma )-1)).$$ Thus, combining Equation \ref{eqn:temp131313} and the above paragraph, we get that \begin{equation}
    v_p(\Delta_{P_{(\alpha_i)/K}})=\frac{\beta}{e_{\base}}(d(d-1)-\sum_{\bar \gamma \in \bar C} \sum_{i \in E_{\bar \gamma}}d_i(\sum_{i \in E_{\bar \gamma}}(d_i/\# \bar \gamma )-1))+\sum_{\bar \gamma \in \bar C} \# \bar \gamma v_p(\Delta_{P_{(\psi_{\gamma_i}(\alpha_i)-\gamma)}/K[\gamma]}). \label{eqn:temp132}
\end{equation} This is the extent of the idea reflected in Figure \ref{fig:recursion}, and this concludes our first step. In summary, what we have done in Equation \ref{eqn:temp132} is express the discriminant of the minimal polynomial of an element $(\alpha_i)$ in terms of \begin{enumerate}
    \item $e_{\base}$, $d_i$, $d$
    \item $\bar C(\{L_i\},K,\{b_i\},\{a_{i,b_i}\})$, and $E_{\bar \gamma}$, $\# \bar \gamma$, $\psi_{\gamma_i}$,
\end{enumerate} where the first item consist of constants that we defined prior to the proof, and the second item consist of objects depending only on the choices of $a_{i,b_i}$ and $L_i$.

As we noted at the beginning of this proof, we will rewrite the formula $$a(\{L_{j_i}\},K,c,\{b_i\})=\mu(\{\{\alpha_i\}_{i \in [m]} \in \prod_{i \in [m]} \fkm_{L_{j_i}}^{b_i}: v_p(\Delta_{P_{(\alpha_i)/K}})=c\})$$$$=\sum_{\substack{i \in I \\ a_{i,b_i} \in T_{f_i,p}}}\mu(\{\{\alpha_i\}_{i \in [m]} \in \prod_{i \in [m]} \fkm_{L_{j_i}}^{b_i}: \text{ for all }i \in I, \alpha_i \in a_{i,b_i}\pi_{L_{j_i}}^{b_i}+\alpha_{i'}, \alpha_i' \in \fkm_{L_{j_i}}^{i+1}, v_p(\Delta_{P_{(\alpha_i)}})=c\},$$ and with Equation \ref{eqn:temp132}, we will now do so. So, we will continue to fix a choice of $\{j_i\}_{i \in [m]}$. Now fix a choice of $(a_{i,b_i})_{i \in I}$. Then, we consider the associated term $$\mu(\{\{\alpha_i\}_{i \in [m]} \in \prod_{i \in [m]} \fkm_{L_{j_i}}^{b_i}: \text{ for all }i \in I, \alpha_i \in a_{i,b_i}\pi_{L_{j_i}}^{b_i}+\alpha_{i'}, \alpha_i' \in \fkm_{L_{j_i}}^{i+1}, v_p(\Delta_{P_{(\alpha_i)}})=c\}.$$ We apply $(\psi_{\gamma_i})_{i=1}^m$ and get 
$$\mu(\{\{\psi_{\gamma_i}(\alpha_i)\}_{i \in [m]} \in \prod_{i \in [m]} \fkm_{\psi_{\gamma_i}(L_{j_i})}^{b_i}: \text{ for all }i \in I, \psi_{\gamma_i}(\alpha_i) = \psi_{\gamma_i}(a_{i,b_i}\pi_{L_{j_i}}^{b_i})+\alpha_{i}', \alpha_i' \in \fkm_{\psi_{\gamma_i}(L_{j_i})}^{i+1}, v_p(\Delta_{P_{(\alpha_i)}})=c\}.$$ Using Equation \ref{eqn:temp132}, we obtain that the above is $$\mu(\{\{\alpha_i'\}_{i \in [m]} \in \prod_{i \in I} \fkm_{\psi_{\gamma_i}(L_{j_i})}^{b_i+1} \prod_{i \notin I} \fkm_{\psi_{\gamma_i}(L_{j_i})}^{b_i}: \sum_{\bar \gamma \in \bar C} \# \bar \gamma v_p(\Delta_{P_{\alpha_i'}/K[\gamma]})=r(c)\}$$$$=\sum_{\substack{c_{\bar \gamma} \in \bbZ_{\geq 0} \\ \sum_{\bar \gamma \in \bar C} c_{\bar \gamma} \# \bar \gamma =r(c)}}\prod_{\bar \gamma \in \bar C}\mu(\{\{\alpha_i'\}_{i \in [m]} \in \prod_{i \in E_{\bar \gamma}} \fkm_{\psi_{\gamma_i}(L_{j_i})}^{b_i+\bb1_{I}(i)}: v_p(\Delta_{P_{\alpha_i'}/K[\gamma]})=c_{\bar \gamma}\}$$$$=\sum_{\substack{c_{\bar \gamma} \in \bbZ_{\geq 0} \\ \sum_{\bar \gamma \in \bar C} c_{\bar \gamma} \# \bar \gamma =r(c)}}\prod_{\bar \gamma \in \bar C}a(\{\psi_{\gamma_i}(L_{j_i})\}_{i \in E_{\bar \gamma}},K[\gamma],c_{\bar \gamma},\{b_i+\bb1_{I}(i)\}_{i \in E_{\bar \gamma}}).$$ Thus, we have $$a(\{L_{j_i}\},K,c,\{b_i\})=\sum_{\substack{i \in I \\ a_{i,b_i} \in T_{f_i,p}}}\sum_{\substack{c_{\bar \gamma} \in \bbZ_{\geq 0} \\ \sum_{\bar \gamma \in \bar C} c_{\bar \gamma} \# \bar \gamma =r(c)}}\prod_{\bar \gamma \in \bar C}a(\{\psi_{\gamma_i}(L_{j_i})\}_{i \in E_{\bar \gamma}},K[\gamma],c_{\bar \gamma},\{b_i+\bb1_{I}(i)\}_{i \in E_{\bar \gamma}}).$$ 

With this, we may now calculate $a(\sigma,e_{\base},f_{\base},c,\{b_i\}_{i=1}^m)$. We have $$a(\sigma,e_{\base},f_{\base},c,\{b_i\}_{i=1}^m) = \frac{1}{\prod_{i \in [m]} \gcd(p^{f_i}-1,e_{i,\rel})}\sum_{j_i \in [\gcd(p^{f_i}-1,e_{i,\rel})]} a(\{L_{j_i}\},K,c,\{b_i\}_{i=1}^m)$$$$=\frac{1}{\prod_{i \in [m]} \gcd(p^{f_i}-1,e_{i,\rel})}\sum_{j_i \in [\gcd(p^{f_i}-1,e_{i,\rel})]} \sum_{\substack{i \in I \\ a_{i,b_i} \in T_{f_i,p}}} $$
\begin{equation} \label{eqn:temp14141}
    \sum_{\substack{c_{\bar \gamma} \in \bbZ_{\geq 0} \\ \sum_{\bar \gamma \in \bar C} c_{\bar \gamma} \# \bar \gamma =r(c)}}\prod_{\bar \gamma \in \bar C}a(\{\psi_{\gamma_i}(L_{j_i})\}_{i \in E_{\bar \gamma}},K[\gamma],c_{\bar \gamma},\{b_i+\bb1_{I}(i)\}_{i \in E_{\bar \gamma}}).
\end{equation} At this point, we have something that somewhat resembles the proposition we are proving. One annoying term in our sum is the $\psi_{\gamma_i}(L_{j_i})$, but now we will fix this. It will turn out to ``equidistribute" nicely over extensions of $K[\gamma]$, as $j_i$ and $a_{i,b_i}$ varies.

We now rearrange the sum. First fix any $\bar C_0$ and partition $\{E_{\bar \gamma}^0\}_{\bar \gamma \in \bar C_0}$ that may show up in the sum. Then, consider the set of pairs $$B=\{(\{j_i\}_{i \in [m]},\{a_{i,b_i}\}_{i \in I}):\bar C_0 = \bar C(\{L_i\},K,\{b_i\},\{a_{i,b_i}\}_{i \in I}),E_{\bar \gamma}^0=E_{\bar \gamma}(\{L_i\},K,\{b_i\},\{a_{i,b_i}\}_{i \in I})\}.$$ The reason to consider such sets $B$ is that for any $(\{j_i\},\{a_{i,b_i}\})$ in this set, Equation \ref{eqn:temp132} is expressed in the same exact way. In other words, $r(c)=r(c,\{E_{\bar \gamma}^0\}_{\bar \gamma \in \bar C_0},\{\# \bar \gamma \}_{\bar \gamma \in \bar C_0})$ is constant on this set. Thus, we may swap the two outer sums as follows $$\sum_{(\{j_i\}_{i \in [m]},\{a_{i,b_i}\}_{i \in I}) \in B} \sum_{\substack{c_{\bar \gamma} \in \bbZ_{\geq 0} \\ \sum_{\bar \gamma \in \bar C} c_{\bar \gamma} \# \bar \gamma =r(c)}}\prod_{\bar \gamma \in \bar C}a(\{\psi_{\gamma_i}(L_{j_i})\}_{i \in E_{\bar \gamma}},K[\gamma],c_{\bar \gamma},\{b_i+\bb1_{I}(i)\}_{i \in E_{\bar \gamma}})$$$$= \sum_{\substack{c_{\bar \gamma} \in \bbZ_{\geq 0} \\ \sum_{\bar \gamma \in \bar C} c_{\bar \gamma} \# \bar \gamma =r(c)}}\sum_{(\{j_i\}_{i \in [m]},\{a_{i,b_i}\}_{i \in I}) \in B}\prod_{\bar \gamma \in \bar C}a(\{\psi_{\gamma_i}(L_{j_i})\}_{i \in E_{\bar \gamma}},K[\gamma],c_{\bar \gamma},\{b_i+\bb1_{I}(i)\}_{i \in E_{\bar \gamma}}),$$ and note that the sum in Equation \ref{eqn:temp14141} is simply the sum of the above over all choices of $B$. Let $$B_{\{j_i^0\}_{i \in [m]}} = \{(\{j_i\}_{i \in [m]},\{a_{i,b_i}\}_{i \in I}): j_i = j_i^0\}.$$ Now we state a lemma that we will prove later. 

\begin{lemma}\label{lem:temp133}
There's a bijection between $B_{\{j_i^0\}_{i \in [m]}}$ and $B_{\{j_i^1\}_{i \in [m]}}$, for any choice of $j_i^0$ and $j_i^1$ such that for all $\bar \gamma \in \bar C_0$ and $i \in E_{\bar \gamma}^0$, $L_{j_i}$ contains some element of $\bar \gamma$. 
\end{lemma}

This lemma is the ``equidistribution" mentioned earlier. A simple way to express this lemma is the fact that as we vary $\{\psi_{\gamma_i}(L_{j_i})\}_{i \in E_{\gamma}}$ up to isomorphism over $(\{j_i\}_{i \in [m]},\{a_{i,b_i}\}_{i \in I}) \in B$, any collections of isomorphism classes of extensions $\{L_i\}_{i \in E_{\gamma}}$ of $K[\gamma]$, such that $f(L_i)=f_i$, $e(L_i)=e_i$ shows up the same number of times. In terms of our proof, we have $$\sum_{(\{j_i\}_{i \in [m]},\{a_{i,b_i}\}_{i \in I}) \in B}\prod_{\bar \gamma \in \bar C}a(\{\psi_{\gamma_i}(L_{j_i})\}_{i \in E_{\bar \gamma}},K[\gamma],c_{\bar \gamma},\{b_i+\bb1_{I}(i)\}_{i \in E_{\bar \gamma}})$$$$=|B|\prod_{\bar \gamma \in \bar C}\sum_{j_i \in [\gcd(p^{f_i}-1,e(K[\gamma]))], i \in E_{\bar \gamma}}a(\{L_{j_i}\}_{i \in E_{\bar \gamma}},K[\gamma],c_{\bar \gamma},\{b_i+\bb1_{I}(i)\}_{i \in E_{\bar \gamma}}).$$ Now let $\sigma$ be partitioned by $\sigma_{\bar \gamma}$, where $\sigma_{\bar \gamma}=(e_i^{f_i})_{i \in E_{\bar \gamma}}$. Now by the induction hypothesis, in particular the dependence on choice of $K$ only via $e(K)$ and $f(K)$, the very inner sum is just $$a(\sigma_{\bar \gamma}, e(K[\gamma]), f(K[\gamma]), c_{\bar \gamma}, \{b_i+\bb1_{I}(i)\}_{i \in E_{\bar \gamma}}).$$ If $\gamma \neq 0$, then we get $$e(K[\gamma])=e(K[\zeta \pi_K^{\beta}])=e(K)\denom(\beta)$$ and $$f(K[\gamma])=f(K)n_O(\gamma)/\denom(\beta)$$ where $\zeta$ is some root of unity with order coprime to $p$. Thus, what we have obtained is that $$\sum_{(\{j_i\}_{i \in [m]},\{a_{i,b_i}\}_{i \in I}) \in B} \sum_{\substack{c_{\bar \gamma} \in \bbZ_{\geq 0}\\ \sum_{\bar \gamma \in \bar C} c_{\bar \gamma} \# \bar \gamma =r(c)}}\prod_{\bar \gamma \in \bar C}a(\{\psi_{\gamma_i}(L_{j_i})\}_{i \in E_{\bar \gamma}},K[\gamma],c_{\bar \gamma},\{b_i+\bb1_{I}(i)\}_{i \in E_{\bar \gamma}})$$$$=|B|\sum_{\substack{c_{\bar \gamma} \in \bbZ_{\geq 0}\\ \sum_{\bar \gamma \in \bar C} c_{\bar \gamma} \# \bar \gamma =r(c)}}\prod_{\bar \gamma \in \bar C}a(\sigma_{\bar \gamma}, e(K)h(\beta,n_O(\gamma)), f(K)\frac{n_O(\gamma)}{h(\beta,n_O(\gamma))}, c_{\bar \gamma}, \{b_i+\bb1_{I}(i)\}_{i \in E_{\bar \gamma}}),$$ where we recall $r(c)=r(c,\{E_{\bar \gamma}^0\}_{\bar \gamma \in \bar C_0},\{\# \bar \gamma \}_{\bar \gamma \in \bar C_0})$

Finally, we complete the rearrangement of the sum $$\frac{1}{\prod_{i \in [m]} \gcd(p^{f_i}-1,e_{i,\rel})}\sum_{j_i \in [\gcd(p^{f_i}-1,e_{i,\rel})]} \sum_{\substack{i \in I \\ a_{i,b_i} \in T_{f_i,p}}} $$$$\sum_{\substack{c_{\bar \gamma} \in \bbZ_{\geq 0} \\ \sum_{\bar \gamma \in \bar C} c_{\bar \gamma} \# \bar \gamma =r(c)}}\prod_{\bar \gamma \in \bar C}a(\{\psi_{\gamma_i}(L_{j_i})\}_{i \in E_{\bar \gamma}},K[\gamma],c_{\bar \gamma},\{b_i+\bb1_{I}(i)\}_{i \in E_{\bar \gamma}}).$$ In the above paragraph, we calculated a subsum after fixing $\bar C_0$ and $\{E^0_{\bar \gamma}\}_{\bar \gamma \in \bar C_0}$ and so we can sum over all possible choices of $\bar C_0$ and $\{E^0_{\bar \gamma}\}_{\bar \gamma \in \bar C_0}$. However, we will instead sum over something more coarse. Given a choice of $\bar C_0$ and $\{E^0_{\bar \gamma}\}_{\bar \gamma \in \bar C_0}$, we have in particular a partition $\{E^0_{\bar \gamma}\}_{\bar \gamma \in \bar C_0}$ of $[m]$, along with for each partition $E^0_{\bar \gamma}$, a number $\# \bar \gamma$. This is how we will rearrange our sum. Thus, we obtain $$\frac{1}{\prod_{i \in [m]} \gcd(p^{f_i}-1,e_{i,\rel})}\sum_{\substack{\text{partitions } \{E_l\}_{l=1}^k\text{ of }[m] \\ \{n_{O,l}\}_{l=1}^k \in \bbZ_{\geq 1}^k }} \# \{(\{j_i\},\{a_{i,b_i}\}_{i \in I}): \{E_l\}_l = \{E_{\bar \gamma}\}_{\bar \gamma \in \bar C}, \{n_{O,l}\}_l = \{\# \bar \gamma \}_{\bar \gamma \in \bar C}\} $$$$ \sum_{\substack{\{c_l\}_{l=1}^k \in \bbZ_{\geq 0}^k \\ \sum_{l} n_{O,l} c_{l}=r(c)}}\prod_{l}a(\sigma_l, e(K)h(\beta,n_{O,l}), f(K)\frac{n_{O,l}}{h(\beta,n_{O,l})}, c_{l}, \{b_i+\bb1_{I}(i)\}_{i \in E_{l}}).$$ We are done after noting that the definition of $P(\sigma,\{E_l\},\{b_i\},\{n_{O,l}\},p)$ is equivalent to $$\frac{1}{\prod_{i \in [m]} \gcd(p^{f_i}-1,e_{i,\rel})}\# \{(\{j_i\},\{a_{i,b_i}\}_{i \in I}): \{E_l\}_l = \{E_{\bar \gamma}\}_{\bar \gamma \in \bar C}, \{n_{O,l}\}_l = \{\# \bar \gamma \}_{\bar \gamma} \in \bar C\}.$$ 


\end{proof}

\begin{proof}[proof of Lemma \ref{lem:temp133}]
Without loss of generality, it suffices to show this when $j_i^0=j_i^1$ except in the first coordinate, where $j_0^0 \neq j_0^1$. Now it's clear that if $0 \notin I$, we already have this bijection, via just replacing $j_0^0$ by $j_0^1$ for each $(\{j_i\},\{a_{i,b_i}\}_{i \in I}) \in B_{\{j_i^0\}_{i=1}^m}$. So, assume that $0 \in I$. Now by assumption, $L_{j_0^0}$ and $L_{j_0^1}$ both contain some conjugate of $\gamma \in C$. This conjugate is of the form $a_{0,b_0} \pi_{L_{j_0^0}}^{b_0}$. Now, since the uniformizer of $L_{j_0^0}$ is of the form $$\pi_{L_{j_0^0}}=(\zeta_{p^{f_0}-1}^{j_0^0}\pi_K)^{1/e_{i,\rel}}=\zeta_{e_{i,\rel}(p^{f_0}-1)}^{j_0^0}\pi_K^{1/e_{i,\rel}},$$ we have that $$\pi_{L_{j_0^0+\frac{e_{i,\rel}}{\gcd(b_0,e_{i,\rel})}n}}^{b_0} = \zeta_{p^{f_0}-1}^{s} \pi_{L_{j_0^0}}^{b_0}$$ for some $s$ depending on $n$. It is also easy to see that $j_0^0 +\frac{e_{i,\rel}}{\gcd(b_0,e_{i,\rel})}n$ are the only such choices. In other words, the choices of $j_0$ so that $L_{j_0}$ contains some conjugate $\gamma$ are precisely fields in the isomorphism classes of $L_{j_0^0+\frac{e_{i,\rel}}{\gcd(b_0,e_{i,\rel})}n}$, which are of the form $L_{j_0^0+\frac{e_{i,\rel}}{\gcd(b_0,e_{i,\rel})}n_1+(p^{f_0}-1)n_2}$. Thus, we see that $\gcd(\frac{e_{i,\rel}}{\gcd(b_0,e_{i,\rel})},p^{f_0}-1)|j_0^1-j_0^0$.

Now let $\gamma^0$ and $\gamma^1$ be the Galois conjugates belonging to $L_{j_0^0}$ and $L_{j_0^1}$, respectively. We will construct a bijection from $B_{\{j_i^0\}_{i \in [m]}}$ to $B_{\{j_i^1\}_{i \in [m]}}$. This bijection will send the pair $(\{j_i^0\}, \{a_{i,b_i}\})$ to $(\{j_i^1\}, \{h_i a_{i,b_i}\})$, where $h_i=1$ except in the case where $i=0$. Thus, we only need to define $h_0$. We will need to find a $h_0$ so that $h_0 a_{0,b_0} \pi_{L_{j_0^1}}^{b_0}$ is a Galois conjugate of $a_{0,b_0} \pi_{L_{j_0^0}}^{b_0}$. Since the uniformizer of $L_{j_0^0}$ is of the form $$\pi_{L_{j_0^0}}=(\zeta_{p^{f_0}-1}^{j_0^0}\pi_K)^{1/e_{i,\rel}}=\zeta_{e_{i,\rel}(p^{f_0}-1)}^{j_0^0}\pi_K^{1/e_{i,\rel}},$$ and similarly for $L_{j_0^1}$, we need to find $h_0 \in T_{f_0,p}$ so that up to Galois conjugates, $$a_{0,b_0} \zeta_{e_{i,\rel}(p^{f_0}-1)}^{b_0j_0^0}\pi_K^{b_0/e_{i,\rel}}=h_0a_{0,b_0} \zeta_{e_{i,\rel}(p^{f_0}-1)}^{b_0j_0^1}\pi_K^{b_0/e_{i,\rel}}.$$ Thus, it suffices to show that up to powers of $\zeta_{e_{i,\rel}}^{b_0}$, $\zeta_{e_{i,\rel}(p^{f_0}-1)}^{b_0j_0^1-b_0j_0^0} \in T_{f_0,p}$. In other words, we ask for the existence of $h_0,h_1$ so that $$b_0(j_0^1-j_0^0) = b_0(p^{f_0}-1)h_1+e_{i,\rel}h_0 \mod e_{i,\rel}(p^{f_0}-1).$$ Now since $$\gcd(\frac{e_{i,\rel}}{\gcd(b_0,e_{i,\rel})},p^{f_0}-1)|j_0^1-j_0^0,$$ we see that $\gcd(e_{i,\rel},b_0(p^{f_0}-1))$ divides the LHS. Since $e_{i,\rel}$ and $b_0(p^{f_0}-1)$ are the coefficients of $h_0,h_1$ on the RHS, we obtain the existence of such $h_0$ and $h_1$. 

\end{proof}

As mentioned after Proposition \ref{prop:induction_a_with_P}, this proposition essentially reduces the calculation of $a(\sigma,e_{\base},f_{\base},c,\{b_i\})$ to $a(\sigma,e_{\base},f_{\base},c,\{b_i+\bb1_I(i)\})$. However, repeatedly applying Proposition \ref{prop:induction_a_with_P} will only keep increasing the elements $\{b_i\}$. Thus, we prove the following lemma will allow us to reduce the calculation of $a(\sigma,e_{\base},f_{\base},c,\{b_i+e_{i,\rel}\})$ to $a(\sigma,e_{\base},f_{\base},c,\{b_i\})$, which will be key in allowing us to close off the induction on degree, as we will soon see.

\begin{lemma}
\label{lem:recursion_reduce_i}
We have the relation $$a(\sigma,e_{\base},f_{\base},c,\{b_i+e_{i,\rel}\}) = p^{-f_{\base}d}a(\sigma,e_{\base},f_{\base},c-\frac{d(d-1)}{e_{\base}},\{b_i\}).$$
\end{lemma}
\begin{proof}
This comes down to the fact that in the definition of discriminant, we have $$\prod_{\alpha_i \neq \alpha_j} (\pi_K\alpha_i-\pi_K\alpha_j)=\pi_K^{d(d-1)}\prod_{\alpha_i \neq \alpha_j} (\alpha_i-\alpha_j).$$ The rest follows by using definition of $a(\sigma,e_{\base},f_{\base},c,\{b_i\})$.
\end{proof}

\subsection{Generating function}

We can express the above recurrences nicely with a generating function. Let $$G(\sigma,e_{\base},f_{\base}, \{b_i\},p)(t) = \sum_{c=0}^\infty a(\sigma,e_{\base},f_{\base}, c, \{b_i\},p)t^c.$$ Similar to our notation for $a$, we will often omit the dependence of $G$ on $p$.

As an analogous example to Example \ref{ex:base_case_a}, we have
\begin{example} \label{ex:base_case_a_GF}
$G((e_{\base}^{f_{\base}}),e_{\base},f_{\base},\{b\})=p^{-bf}$.
\end{example} This will be our base case, and we will inductively calculate $G(\sigma,e_{\base},f_{\base}, \{b_i\})(t)$ using Proposition \ref{prop:induction_a_with_P} and Lemma \ref{lem:recursion_reduce_i}.

To simplify our equations, define $$\inc^1(\{b_i\})=\inc(\sigma,e_{\base},f_{\base},\{b_i\}) = \{b_i+\bb1_{I}(i)\},$$ where we recall, $I$ consists of $i$ so that $b_i/e_{i,\rel}$ is minimum among such terms. Let $\inc^n(\{b_i\})=\inc(\inc^{n-1}(\{b_i\}))$. Define \begin{align*}
    H(\sigma,e_{\base},f_{\base}, \{b_i\}) &= \left(\sum_{\substack{\text{partitions } \{E_l\}_{l=1}^k \text{ of }[m] \\ \{n_{O,l}\}_{l=1}^k \in \bbZ_{\geq 1}^k \\
(\{E_l\},\{n_{O,l}\}) \neq (\{[m]\},\{1\})}} P(\sigma,\{E_l\},\{b_i\},\{n_{O,l}\},p)\right)\\
&\times t^{\frac{1}{e_{\base}}(d(d-1)-\sum_l \sum_{i \in E_l}d_i(\sum_{i \in E_l}(d_i/n_{O,l})-1))\beta} \\
&\times \prod_{l=1}^k G(\sigma_l, e(K)h(\beta,n_{O,l}), f(K)\frac{n_{O,l}}{h(\beta,n_{O,l})}, \{b_i+\bb1_{I}(i)\}_{i \in E_{l}})(t^{n_{O,l}})
\end{align*} From now on, we will often abbreviate $G(\sigma,e_{\base},f_{\base}, \{b_i\})(t)$ with $G(\{b_i\})(t)$, and similarly for $H(\{b_i\})(t)$. With these notations, we get the following as a reformulation of Proposition \ref{prop:induction_a_with_P}.
\begin{cor} \label{cor:GF_induction}
$$G(\{b_i\})(t) = P(\sigma, \{[m]\},\{b_i\},\{1\},p)G(\inc(\{b_i\}))(t)+H(\{b_i\})(t)$$
\end{cor}

Also, as a reformulation of Lemma \ref{lem:recursion_reduce_i}, we get that 
\begin{cor}\label{cor:GF_reduce_i}
$$G(\{b_i+e_{i,\rel}\})(t) = p^{-f_{\base}d}t^{\frac{d(d-1)}{e_{\base}}}G(\{b_i\})(t)$$
\end{cor}

Next, we will use these two reformulations to inductively calculate $(\sigma,e_{\base},f_{\base}, \{b_i\})$. Now fix $\{b_i\}$ and let $k_{lim}=\max_i\lceil \frac{b_i}{e_{i,\rel}}\rceil$. It's not hard to see that $\inc^n(\{b_i\})$ will eventually give us the constant sequence $\{k_{lim}e_{i,\rel}\}$. Let $n_{\{b_i\}}$ be the smallest $n$ so that $\inc^n(\{b_i\}) = \{k_{lim}e_{i,\rel}\}$. Also, let $n_{rec}$ be the $n$ so that $\inc^n(\{0\}_{i=1}^m)=\{e_{i,\rel}\}$. Then, we have the following formula.
\begin{prop}
\label{prop:generating_function_recursion}
$$G(\{b_i\})(t)=\sum_{n=0}^{n_{\{b_i\}}-1}H(\inc^n(\{b_i\}))(t)+\left(p^{-f_{\base}d}t^{\frac{d(d-1)}{e_{\base}}}\right)^{k_{lim}}G(\{0\}_{i=1}^m)(t),$$ where $$G(\{0\}_{i=1}^m)(t)=\frac{1}{1-p^{f_{\base}(1-d)}t^{\frac{d(d-1)}{e_{\base}}}}\left(H(\{0\}_{i=1}^m)(t)+p^{f_{\base}}\sum_{n=1}^{n_{rec}-1} H(\inc^n(\{0\}_{i=1}^m))(t)\right).$$
\end{prop}
\begin{proof}
First, we use Corollary \ref{cor:GF_induction} repeatedly and get that $G(\{b_i\})=G(\{k_{lim}e_{i,\rel}\})+\sum_{n=0}^{n_{\{b_i\}}-1}H(\inc^n(\{b_i\}))(t)$, since in these cases $P(\sigma, \{[m]\},\{b_i\},\{1\},p)=1$ by Theorem \ref{thm:aiji_count}. Now by applying Corollary \ref{cor:GF_reduce_i} $k_{lim}$ times to $G(\{k_{lim}e_{i,\rel}\})$, we get that it suffices to show $$G(\{0\}_{i=1}^m)(t)=\frac{1}{1-p^{f_{\base}(1-d)}t^{\frac{d(d-1)}{e_{\base}}}}\left(H(\{0\}_{i=1}^m)(t)+p^{f_{\base}}\sum_{n=1}^{n_{rec}-1} H(\inc^n(\{0\}_{i=1}^m))(t) \right).$$ Indeed, applying Corollary \ref{cor:GF_induction} $n_{rec}$ number of times here, we get $$G(\{0\}_{i=1}^m)(t)=H(\{0\}_{i=1}^m)(t)+p^{f_{\base}}\sum_{n=1}^{n_{rec}-1} H(\inc^n(\{0\}_{i=1}^m))(t)+p^{f_{\base}}G(\{e_{i,\rel}\})(t),$$ since $P(\sigma, \{[m]\},\{0\}_{i=1}^m,\{1\},p)=p^{f_{\base}}$. Finally, since $G(\{e_{i,\rel}\})(t)=p^{-f_{\base}d}t^{\frac{d(d-1)}{e_{\base}}}G(\{0\}_{i=1}^m)(t)$ by Corollary \ref{cor:GF_reduce_i}, we can solve for $G(\{0\}_{i=1}^m)(t)$ and obtain $$G(\{0\}_{i=1}^m)(t)=\frac{1}{1-p^{f_{\base}(1-d)}t^{\frac{d(d-1)}{e_{\base}}}}\left(H(\{0\}_{i=1}^m)(t)+p^{f_{\base}}\sum_{n=1}^{n_{rec}-1} H(\inc^n(\{0\}_{i=1}^m))(t)\right).$$
\end{proof}

\section{Proof of Main Result}\label{sec:proof_of_main_result}

For this section, we use the same notations as we did for the section \ref{sec:generating_function}. In particular, we use \textit{absolute} ramification indices and inertia degrees. 

Let $q=p^{f_{\base}}$. Recall that we wish to show that given a tamely ramified splitting type $\sigma = (e_1^{f_1}, \dots, e_m^{f_m})$ over a base field with ramification index $e_{\base}$ and inertia degree $f_{\base}$, that $$\frac{q^{\frac{\sum(e_{i,\rel}-1)f_{i,\rel}}{2}}}{\prod_{i=1}^m f_{i,\rel}\gcd(p^{f_i}-1,e_{i,\rel})}\sum_{\substack{L=\prod_i L_{j_i}\\ j_i \in [\gcd(p^{f_i}-1,e_{i,\rel})]-1}}\int_{\alpha \in \prod_i \fkm_{L_{j_i}}^{b_i}} q^{-\frac{1}{2}v_K(\Delta_{P_{\alpha/K}})} d\lambda_L$$ is a rational function in terms of $q$ and depends on $K$ only through $e(K)$ and $f(K)$.

\begin{proof}[proof of Theorem \ref{thm:main}]

Let $q=p^{f_{\base}}$. It's clear from the definition that $$\frac{1}{\prod_{i=1}^m \gcd(p^{f_i}-1,e_{i,\rel})}\sum_{\substack{L=\prod_i L_{j_i}\\ j_i \in [\gcd(p^{f_i}-1,e_{i,\rel})]}}\int_{\alpha \in \prod_i \fkm_{L_{j_i}}^{b_i}} p^{-\frac{1}{2}v_K(\Delta_{P_{\alpha/K}})} d\lambda_L$$ is precisely $G(\sigma, e_{\base}, f_{\base}, \{b_i\})(q^{-e(K)/2})$. Thus, part 1 of the theorem on the dependence on $K$ via $e(K)$ and $f(K)$ is a corollary of Proposition \ref{prop:induction_a_with_P}, since $G$ is a generating function with coefficients of the form $a(\sigma,e_{\base},f_{\base},c,\{b_i\},p)$.

Now we start on part 2 of the theorem on the fact that it's a rational function in terms of $q$. We again have our base case (Example \ref{ex:base_case_a_GF}) $G((e^f),e,f,b)=p^{-bf}=q^{-b}$, which is a rational function in terms of $q$. Next, we use Proposition \ref{prop:generating_function_recursion}. Now, evaluating at $t=q^{-e_{\base}/2}$, we obtain that $$q^{\frac{\sum(e_{i,\rel}-1)f_{i,\rel}}{2}}G(\{b_i\})(q^{-e_{\base}/2})$$ is a rational function of $q$ as long as $$q^{\frac{\sum(e_{i,\rel}-1)f_{i,\rel}}{2}}H(\{b_i\})(q^{-e_{\base}}/2)$$ are all rational functions. Unraveling the definition of $H(\{b_i\})(t)$, we need to show that when evaluated $t=q^{-e_{\base}/2}$, the product of the following three terms \begin{enumerate}
    \item $P(\sigma,\{E_l\},\{b_i\},\{n_{O,l}\},p)$
    \item $t^{\frac{1}{e_{\base}}(d(d-1)-\sum_l \sum_{i \in E_l}d_i(\sum_{i \in E_l}(d_i/n_{O,l})-1))\beta}$
    \item $\prod_l G(\sigma_l, e(K)h(\beta,n_{O,l}), f(K)\frac{n_{O,l}}{h(\beta,n_{O,l})}, c_l, \{b_i+\bb1_{I}(i)\}_{i \in E_{l}})(t^{n_{O,l}})$
\end{enumerate} is a rational function in terms of $q^{1/2}$, and moreover the exponent of this rational function has the same parity as $\sum(e_{i,\rel}-1)f_{i,\rel}$.

For the first item, we have by Theorem \ref{thm:aiji_count} that it is a rational function in terms of $q$. So, we need that the product of second and third item to be a rational function in terms of $q^{1/2}$ after evaluating at $t=q^{-e_{\base}/2}$, and the exponent of this rational function has the same parity as $\sum(e_{i,\rel}-1)f_{i,\rel}$. To do this, we will evaluate at $t=q^{-e_{\base}}$, show that all powers of $q$ are integral, then track the parity of powers of $q$.

We consider now the second item. Evaluating at $q^{-e_{\base}}$ gives us an exponent of $$-(d(d-1)-\sum_l \sum_{i \in E_l}d_i(\sum_{i \in E_l}(d_i/n_{O,l})-1))\beta).$$ We rewrite this as $$-\sum_l \sum_{i \in E_l}d_i(d-\sum_{i \in E_l}(d_i/n_{O,l}))\beta$$ and show each term is integral. Now $\denom(\beta)|n_{O,l}$ for all $l$ except in the case where $\denom(\beta) \neq 1$ and $n_{O,l}=1$. Since we must have $n_{O,l}|d_i$, to show integrality, we only have to consider the case with $\denom(\beta) \neq 1$ and $n_{O,l}=1$. In this case, $d-\sum_{i \in E_l}(d_i/n_{O,l}))$ is precisely $\sum_{i \notin E_l} d_i$, and since $\denom(\beta)$ divides every such $d_i$, we get that $$-\sum_l \sum_{i \in E_l}d_i(d-\sum_{i \in E_l}(d_i/n_{O,l}))\beta$$ is indeed an integer.

For the third item, we have by induction hypothesis that evaluating $$G(\sigma_l, e(K)h(\beta,n_{O,l}), f(K)\frac{n_{O,l}}{h(\beta,n_{O,l})}, c_l, \{b_i+\bb1_{I}(i)\}_{i \in E_{l}})$$ at $q^{e_{\base} n_{O,l}}$ gives us integer exponents in terms of $q^{\frac{n_{O,l}}{h(\beta,n_{O,l})}}$.

Now we track the parity of the exponents. This is purely an exercise in casework, unravelling definitions, and simplifying until we see that the exponents have even parity. By inductive hypothesis, the parity of the exponent of $q^{\frac{n_{O,l}}{h(\beta,n_{O,l})}}$ in $$G(\sigma_l, e(K)h(\beta,n_{O,l}), f(K)\frac{n_{O,l}}{h(\beta,n_{O,l})}, c_l, \{b_i+\bb1_{I}(i)\}_{i \in E_{l}})$$ evaluated at $q^{e_{\base} n_{O,l}}$ is precisely the parity of $\sum_{i \in E_l}(\frac{e_{i,\rel}}{h(\beta,n_{O,l})}-1)\frac{f_{i,\rel}h(\beta,n_{O,l})}{n_{O,l}}$. Thus, the parity of the exponent in terms of $q$ is the parity of $$\sum_{i \in E_l}(\frac{e_{i,\rel}}{h(\beta,n_{O,l})}-1)f_{i,\rel}.$$

Thus, we must show that $$\sum(e_{i,\rel}-1)f_{i,\rel}-\sum_l \sum_{i \in E_l}d_i(d-\sum_{i \in E_l}(d_i/n_{O,l}))\beta+\sum_l \sum_{i \in E_l}(\frac{e_{i,\rel}}{h(\beta,n_{O,l})}-1)f_{i,\rel}$$ is even. We have a series of equalities modulo $2$. $$\sum(e_{i,\rel}-1)f_{i,\rel}-\sum_l \sum_{i \in E_l}d_i(d-\sum_{i \in E_l}(d_i/n_{O,l}))\beta+\sum_l \sum_{i \in E_l}(\frac{e_{i,\rel}}{h(\beta,n_{O,l})}-1)f_{i,\rel}$$$$=\sum_l\sum_{i \in E_l}d_i(d-\sum_{i \in E_l}(d_i/n_{O,l}))\beta+\sum_l\sum_{i \in E_l}(\frac{e_{i,\rel}}{h(\beta,n_{O,l})}-e_{i,\rel})f_{i,\rel}$$$$=\sum_l\sum_{i \in E_l}f_{i,\rel}(e_{i,\rel}(d-\sum_{i \in E_l}(d_i/n_{O,l}))\beta+(\frac{e_{i,\rel}}{h(\beta,n_{O,l})}-e_{i,\rel})).$$ First of all, modulo $2$, we may remove all $i$ such that $f_{i,\rel}$ is even. So, we reduce to the case where $f_{i,\rel}$ are all odd.

In the case where $f_{i,\rel}$ are all odd, we have that $d_i \equiv e_{i,\rel} \mod 2$. So, we may rewrite our sum as $$\sum_l\sum_{i \in E_l}e_{i,\rel}(\sum_i e_{i,\rel}-\sum_{i \in E_l}(d_i/n_{O,l}))\beta+(\frac{e_{i,\rel}}{h(\beta,n_{O,l})}-e_{i,\rel})).$$ Now for $n_{O,l} \neq 1$, it is actually the degree of a subextension $K'/K$ of some $L_{j_i}$, where $K'$ is of the form $K' = K[a_{i,b_i}\pi_{L_{j_i}}^{b_i}]$. With this notation, $\denom(\beta)=e(K'/K)$, and $n_{O,l}=\denom(\beta)f(K'/K)$. Now note that since we're in the case where $f_{i,\rel}$ are all odd, modulo $2$, we can assume $n_{O,l}=\denom(\beta)$. Now suppose there are no $l$ such that $n_{O,l}=1$. Then, we may rewrite our sum again as $$\sum_l\sum_{i \in E_l}(e_{i,\rel}(\sum_i e_{i,\rel}-\sum_{i \in E_l}\frac{e_{i,\rel}}{\denom(\beta)})\beta+(\frac{e_{i,\rel}}{\denom(\beta)}-e_{i,\rel})).$$ Now let $\num(\beta)$ denote the numerator of $\beta$ in the lowest form. Then, we get $$\num(\beta)\sum_l \sum_{i \in E_l} \frac{e_{i,\rel}}{\denom(\beta)}(\sum_i e_{i,\rel}-\sum_{i \in E_l} \frac{e_{i,\rel}}{\denom(\beta)})+\sum_i (\frac{e_{i,\rel}}{\denom(\beta)}-e_{i,\rel}).$$ Multiplying the first term out, we get $$\num(\beta)\left(\sum_i \frac{e_{i,\rel}}{\denom(\beta)}\sum_i e_{i,\rel}-\sum_l\left(\sum_{i \in E_l} \frac{e_{i,\rel}}{\denom(\beta)}\right)^2\right)+\sum_i (\frac{e_{i,\rel}}{\denom(\beta)}-e_{i,\rel}).$$ But modulo $2$, $x^2=x$, and so our expression simplifies to  $$\num(\beta)\left(\sum_i \frac{e_{i,\rel}}{\denom(\beta)}(\sum_i e_{i,\rel}-1)\right)+\sum_i (\frac{e_{i,\rel}}{\denom(\beta)}-e_{i,\rel}).$$ Now we have two cases. If $\denom(\beta)$ were odd, then we immediately see that both terms are $0 \mod 2$. If $\denom(\beta)$ were even, then $\num(\beta)$ must be odd. In this case, our expression simplifies to $$\sum_i e_{i,\rel} (\sum_i \frac{e_{i,\rel}}{\denom(\beta)}-1).$$ But we know $\denom(\beta)|e_{i,\rel}$ for each $i$, and so this term is even as well.

Now we are left with the case where there is some $l$ with $n_{O,l}=1$, say $n_{O,1}$. All of the above reasoning still applies, and the term $$\sum_{i \in E_l}e_{i,\rel}(\sum_i e_{i,\rel}-\sum_{i \in E_l}(d_i/n_{O,l}))\beta+(\frac{e_{i,\rel}}{h(\beta,n_{O,l})}-e_{i,\rel})),$$ where $l=1$ simplifies to $$\sum_{i \in E_1}e_{i,\rel}\sum_{i \notin E_1}e_{i,\rel}\beta.$$ Thus, we wish to show that $$\sum_{i \in E_1}e_{i,\rel}\sum_{i \notin E_1}e_{i,\rel}\beta+ \num(\beta)\left(\sum_{i \notin E_1} \frac{e_{i,\rel}}{\denom(\beta)}(\sum_i e_{i,\rel}-1)\right)+\sum_{i \notin E_1} (\frac{e_{i,\rel}}{\denom(\beta)}-e_{i,\rel})$$ is even. This expression is equivalent to the following $$\num(\beta)\left(\sum_{i \notin E_1} \frac{e_{i,\rel}}{\denom(\beta)}(\sum_i e_{i,\rel}-1+\sum_{i \in E_1} e_{i,\rel})\right)+\sum_{i \notin E_1} (\frac{e_{i,\rel}}{\denom(\beta)}-e_{i,\rel})$$$$=\num(\beta)\left(\sum_{i \notin E_1} \frac{e_{i,\rel}}{\denom(\beta)}(\sum_{i \notin E_1} e_{i,\rel}-1)\right)+\sum_{i \notin E_1} (\frac{e_{i,\rel}}{\denom(\beta)}-e_{i,\rel}).$$ As before, we separate into cases where $\denom(\beta)$ is odd and $\denom(\beta)$ is even, and see that in both cases, this expression is even.
\end{proof}

\section{Proof of Theorem \ref{thm:reduce_to_OL}} \label{sec:red_to_OL_thm_proof}

First, we setup and recall some notation. For $a$ a nonnegative integer, let $R[x]_a^1$ denote the monic polynomials of degree $a$ over $R$. In the case where $R=\cO_K$, we identify $\cO_K[x]_a^1$ with $\cO_K^a$ via recording the coefficients of every term besides the $x^a$ term. Thus, we also get a measure on this space, which we'll call $\mu_a^1$ (recall that $\mu_a$ is the name of the measure on $\Omega_a$). Let $\cO_K[x]_a^2 \subset \Omega_a$ denote the polynomials $P$ of degree $a$ over $\cO_K$ that reduce to $\bar P = 1$ in the residue field $\kappa_{K}=\cO_K/\fkm_K$. Let $q=p^f$, where $f$ is the absolute inertia degree of $K$. As before, let $w(d)=w_{f(K)}(d)=\frac{q^{d+1}-q^{d}}{q^{d+1}-1}$.

Let $L= \prod_{i=1}^m L_i$ be any \'etale extension of $K$. For $U \subset L$, denote $\rho(U,K)$ to be the proportion of degree $[L:K]$ polynomials (in $\Omega_{[L:K]}$ with a new root in $U$. We have by linearity of expectation that $$\rho(L,K) = \sum_{A \subset [m]} \frac{1}{|\Aut(A)|}\rho(\prod_{i \in A} \cO_{L_i} \times \prod_{i \notin A} L_i \setminus \cO_{L_i},K)$$ where $\Aut(A)$ is the number of $A'$ so that $L_A \cong L_{A'}$. This is equal to $\left|\frac{\Aut(L)}{\Aut(L_A) \Aut(L_{A^c})}\right|$ (since $\Aut(L)$ acts transitively on the set of such $A'$, and the stabilizer of $A$ is precisely $\Aut(L_A) \times \Aut(L_{A^c})$). Now, by the same argument in Lemma \ref{lem:temp1}, we get that $\rho(U,K)=\frac{\bbE[Z^{\new}_{U,L}]}{\Aut(L)}$. Thus, it suffices to show that $$\rho(\prod_{i \in A} \cO_{L_i} \prod_{i \notin A} L_i \setminus \cO_{L_i}, K)=\frac{w([L:K])}{w([L_A:K])w([L_{A^c}:K])}\rho(\cO_{L_A},L_A)\rho(\fkm_{L_{A^c}},L_{A^c}).$$ The central idea to this equivalence is given in \cite[Lemma 2.9]{bhargava_cremona_fisher_gajović_2022}. In the case where $L$ is an honest field extension, the central idea manifests as the fact that polynomials in $\Omega_d$ whose roots are in $\fkm_L$ are identified with the polynomials whose roots are in $L \setminus \cO_L$, via the reversal map $P \to x^d P(1/x)$, and so since the reversal map in this case is just a permutation of the coordinates, the measure of the two sets of polynomials should be the same.

\begin{lemma}[Bhargava, Cremona, Fisher, Gajovi\'{c}]
Let $a \geq b \geq 0$. The multiplication map $$\psi:\cO_K[x]_b^1 \times \cO_K[x]_{a-b}^2 \to \{P \in \Omega_a: \bar P \in \kappa_K[x]^1_b\}$$ is a measure preserving bijection.
\end{lemma}

Note that the authors proved this over the $p$-adics, but it extends to all extensions as well. Applying this lemma, we get that $$\{P_1 \in \cO_K[x]_{[L_A:K]}^1: P_1 \text{ has a new root in } \cO_{L_A}\} \times \{P_2 \in \cO_K[x]_{[L_{A^c}:K]}^2: P_2 \text{ has a new root in } \prod_{i \notin A}L_i \setminus \cO_{L_i}\},$$ which is a subset of $\cO_K[x]_{[L_A:K]}^1 \times \cO_K[x]_{[L_{A^c}:K]}^2$, has the same measure as $$\{P \in \Omega_{[L:K]}: \bar P \in \kappa_K[x]^1_b, P \text{ has a new root in } \cO_{L_A} \times \prod_{i \notin A}L_i \setminus \cO_{L_i}\},$$ which is a subset of $\{P \in \Omega_{[L:K]}: \bar P \in \kappa_K[x]^1_{[L_{A}:K]}\}.$

Thus, it suffices for us to show the following \begin{enumerate}
    \item $$\rho(\cO_{L_A},K)=w([L_A:K])\mu_{[L_A:K]}^1(\{P_1 \in \cO_K[x]_{[L_A:K]}^1: P_1 \text{ has a new root in } \cO_{L_A}\}),$$
    \item $$\rho(\fkm_{L_{A^c}},L_{A^c})=qw([L_{A^c}:K]))\mu_{[L_{A^c}:K]}(\{P_2 \in \cO_K[x]_{[L_{A^c}:K]}^2: P_2 \text{ has a new root in } \prod_{i \notin A}L_i \setminus \cO_{L_i}\}),$$
    \item and $$\rho(\prod_{i \in A} \cO_{L_i} \prod_{i \notin A} L_i \setminus \cO_{L_i}, K)$$$$=qw([L:K])\mu_{[L:K]}(\{P \in \Omega_{[L:K]}: \bar P \in \kappa_K[x]^1_b, P \text{ has a new root in } \cO_{L_A} \times \prod_{i \notin A}L_i \setminus \cO_{L_i}\}).$$
\end{enumerate}

We have a surjective map $$\varphi: \{\sum_{i=0}^{[L_A:K]}a_i x^i \in \Omega_{[L_A:K]}: v_p(a_{[L_A:K]}) \leq v_p(a_i), i=0, \dots, [L_A:K]-1\} \to \cO_K[x]_{[L_A:K]}^1$$ via dividing by the coefficient of $x^{[L_A:K]}$. One can check on basic opens that the pushforward of $\mu_{[L_A:K]}$ via this map is precisely $w([L_A:K])\mu_{[L_A:K]}^1$. Now, $$\varphi^{-1}(\{P_1 \in \cO_K[x]_{[L_A:K]}^1: P_1 \text{ has a new root in } \cO_{L_A}\}),$$ is precisely the set of which $\rho(\cO_{L_A},K)$ is the measure. Hence, we have our first item.

The second and third item are similar, and so we will just prove the second item. Let $$S_{\beta}=\{P_2 \in \cO_K[x]_{[L_{A^c}:K]}: P_2 \text{ has a new root in } \prod_{i \notin A}L_i \setminus \cO_{L_i}, \bar P_2 = \beta\},$$ so that $S_1$ is our set on the RHS of item 2. Then, we see by multiplying by units, that $\mu_{[L_A^c:K]}(S_1)=\mu_{[L_A^c:K]}(S_{\beta})$ for any $\beta \in \kappa_K \setminus \{0\}$. Now let $$S=\{P_2 \in \cO_K[x]_{[L_{A^c}:K]}: P_2 \text{ has a new root in } \prod_{i \notin A}L_i \setminus \cO_{L_i}\},$$ so that $$\rho(\prod_{i \notin A} L_i \setminus \cO_{L_i},K)=\mu_{[L_{A^c}:K]}(S).$$ Then, $S = \pi_K S\sqcup \bigsqcup_{\beta \in \kappa_K \setminus \{0\}} S_{\beta}$. Hence, $$(1-q^{-([L_{A^c}:K]+1)})\rho(\prod_{i \notin A} L_i \setminus \cO_{L_i},K)=(q-1)\mu_{[L_{A^c}:K]}(S_1).$$ Thus, we get $$\rho(\prod_{i \notin A} L_i \setminus \cO_{L_i},K)=qw([L_{A^c}:K])\mu_{[L_{A^c}:K]}(S_1),$$ as desired.

\section{Computation of $P(\sigma,\{E_l\},\{b_i\},\{n_{O,l}\},p)$} \label{sec:counting_choices}

In this section, our goal is to calculate $P(\sigma,\{E_l\},\{b_i\},\{n_{O,l}\},p)$, which was defined to be $\frac{1}{\prod_{i \in [m]} \gcd(p^{f_i}-1,e_{i,\rel})}$ times the number of elements $$(\{j_i\}_{i=1}^m, \{a_{i,b_i}\}_{i \in I}) \in \prod_{i=1}^m [\gcd(p^{f_i}-1,e_{i,\rel}] \times \prod_{i \in I} T_{f_i,p}$$ satisfying the following three conditions: 
\begin{enumerate}
    \item For all $i_1,i_2 \in I$, $i_1,i_2$ belong to the same $E_l$ if and only if $a_{i_1,b_{i_1}}\pi_{L_{j_{i_1}}}^{b_{i_1}}$ and $a_{i_2,b_{i_2}}\pi_{L_{j_{i_2}}}^{b_{i_2}}$ are Galois conjugates.
    \item If $i \in I$, then $n_O(a_{i,b_{i}}\pi_{L_{j_{i}}}^{b_{i}})=n_{O,l}$.
    \item If $I \neq [m]$, then there is some $E_{l_0}$ containing all of $[m] \setminus I$. Moreover, for any $i \in E_{l_0} \cap I$, $a_{i,b_i}=0$, and so in particular, $n_{O,l_0}=1$.
\end{enumerate}

Now suppose $P(\sigma,\{E_l\},\{b_i\},\{n_{O,l}\},p)$ is nonzero. Then, first of all, since for $i \in I$, we have that $e(K[a_{i,b_i}\pi_{L_{j_i}}^{b_i}]/K)$ is $\denom(\beta)$ unless $a_{i,b_i}=0$, we must have that unless $n_{O,l}=1$ and $\denom(\beta) \neq 1$ (in which case $a_{i,b_i}\pi_{L_{j_i}}^{b_i}=0$), $\denom(\beta)|n_{O,l}$ for all $l$ and $\frac{n_{O,l}}{\denom(\beta)}|f_{i,\rel}$ for all $i \in E_l$. Second of all, we must have that for any given $k$, the number of elements $E_l$ of the partition such that $n_{O,l}=k$ is at most the number of possible Galois orbits of size $k$ among elements of the form $a_{i,b_i}\pi_{L_{j_i}}^{b_i}$. We will shortly see that the number of such Galois orbits is $P(f_{\base},k/\denom(\beta),p)$, where $$P(f_{\base},k',p)=\frac{1}{k'}((p^{f_{\base}k'}-1)-\sum_{\ell_1 \text{ prime}, \ell_1|k'} (p^{f_{\base}k'/\ell_1}-1)+\sum_{\ell_1 \neq \ell_2 \text{ prime}, \ell_i|k'} (p^{\frac{f_{\base}k'}{\ell_1 \ell_2}}-1)-\dots).$$ Thus, $P(\sigma,\{E_l\},\{b_i\},\{n_{O,l}\},p)=0$ unless all of following conditions are satisfied: \begin{enumerate}
    \item For all $l$ such that $n_{O,l} \neq 1$, $\denom(\beta)|n_{O,l}$.
    \item For all $i,l$ such that $i \in E_l$ and $n_{O,l} \neq 1$, $\frac{n_{O,l}}{\denom(\beta)} | f_{i,\rel}$.
    \item $\#\{l: n_{O,l}=k\} \leq P(f_{\base},k/\denom(\beta),p)$.
    \item If $\denom(\beta) \neq 1$, then there is at most one $l_0$ so that $n_{O,l_0}=1$.
    \item If $I \neq [m]$, then there is some $l_0$ such that $E_{l_0}$ contains all of $[m] \setminus I$. 
\end{enumerate}

For the rest of this section, assume all of these conditions hold. Let $l_0$ be defined as in the conditions above. Now let $F_k = \{i \in [m]: \exists l, n_{O,l}=k, i \in E_l\}.$ Let $P_k(\sigma,\{E_l\},\{b_i\},\{n_{O,l}\},p)$ be the number of elements $$(\{j_i\}_{i \in F_k}, \{a_{i,b_i}\}_{i \in I \cap F_k}) \in \prod_{i \in F_k} [\gcd(p^{f_i}-1,e_{i,\rel}] \times \prod_{i \in F_k \cap I} T_{f_i,p}$$ satisfying the following conditions: 
\begin{enumerate}
    \item For all $i_1,i_2 \in F_k \cap I$, $i_1,i_2$ belong to the same $E_l$ if and only if $a_{i_1,b_{i_1}}\pi_{L_{j_{i_1}}}^{b_{i_1}}$ and $a_{i_2,b_{i_2}}\pi_{L_{j_{i_2}}}^{b_{i_2}}$ are Galois conjugates.
    \item If $i \in I$, then $n_O(a_{i,b_{i}}\pi_{L_{j_{i}}}^{b_{i}})=k$.
    \item If $k=1$ and $I \neq [m]$, then for any $i \in E_{l_0} \cap I$, $a_{i,b_i}=0$.
\end{enumerate} Since these conditions are independent as $k$ varies, we see that $$P(\sigma,\{E_l\},\{b_i\},\{n_{O,l}\},p)=\prod_k P_k(\sigma,\{E_l\},\{b_i\},\{n_{O,l}\},p),$$ as $k$ varies over integers for which $F_k$ is nonempty. Thus, it suffices to calculate $P_k$. Now $$P_k(\sigma,\{E_l\},\{b_i\},\{n_{O,l}\},p) = P((e_i^{f_i})_{i \in F_k}, \{E_l\}_{l, n_{O,l}=k}, \{b_i\}_{i \in F_k}, \{k\}_{l, n_{O,l}=k},p).$$ Thus, it suffices for us to calculate $P(\sigma,\{E_l\},\{b_i\},\{n_{O,l}\},p)$, in the case where all $n_{O,l}$ are equal to $k$, which we will now proceed to do in the case where $k \neq 1$. The case where $k=1$ is similar, but just requires more casework.

Now we start with the simplest case, where $\sigma$ is of the form $(e_1^{f_1})$. Thus, $m=1$ and we have that the only partition of $[1]$ consist of one set $E_1=\{1\}$.

\begin{lemma}
\label{lem:nO_simple_case}
Suppose $\denom(\beta)|k$ (here $\beta=\min_i \frac{b_i}{e_{i,\rel}} = \frac{b_1}{e_{1,\rel}}$, and so $\denom(\beta)=\frac{e_{1,\rel}}{\gcd(b_1,e_{1,\rel})}$), then $$\#\{(a_{1,b_1},j_1) \in (T_{f_1,p} \setminus \{0\}) \times [\gcd(p^{f_1}-1,e_{1,\rel})], n_O(a_{1,b_1}\pi_{L_{j_1}}^{b_1})=k\}$$$$=\gcd(p^{f_1}-1,e_{1,\rel})\frac{k}{\denom(\beta)}P(f_{\base},\frac{k}{\denom(\beta)},p).$$
\end{lemma} 

Hence, in this simple case, we have that $$P((e_1^{f_1}),\{E_1\},\{b_1\},\{k\},p)=\gcd(p^{f_1}-1,e_{1,\rel})\frac{k}{\denom(\beta)}P(f_{\base},\frac{k}{\denom(\beta)},p)$$ if $k \neq 1$ (since $a_{1,b_1}$ cannot be $0$ if $n_{O,1}=k \neq 1$). Next, we do the case where $\sigma=(e_1^{f_1},e_2^{f_2})$, and the general case will follow by similar reasoning. This case comes down to the following.

\begin{lemma}
\label{lem:same_gal_conj_count}
Suppose $b_1/e_{1,\rel}=b_2/e_{2,\rel}$ (and so $\beta=\frac{b_1}{e_{1,\rel}}=\frac{b_2}{e_{2,\rel}}$ and $\denom(\beta)=\frac{e_{1,\rel}}{\gcd(b_1,e_{1,\rel})}=\frac{e_{2,\rel}}{\gcd(b_2,e_{2,\rel})}$). Given an element $a_{1,b_1}\pi_{L_{j_1}}^{b_1}$ with $k$ Galois conjugates, the number of elements $a_{2,b_2} \pi_{L_{j_2}}^{b_2}$ such that it is a Galois conjugate of $a_{1,b_1}\pi_{L_{j_1}}^{b_1}$ is $\frac{k}{\denom(\beta)}\gcd(p^{f_2}-1,e_{2,\rel})$.
\end{lemma}

First of all, note that combining Lemmas \ref{lem:nO_simple_case} and \ref{lem:same_gal_conj_count} shows that the number of Galois orbits of size $k$ among elements of the form $a_1 \pi_{L_1}^{b_1}$ is precisely $P(f_{\base},k/\denom(\beta),p)$ (here $\beta = \frac{b_1}{e_{1,\rel}}$) by setting $b_1=b_2$ and $e_{1,\rel}=e_{2,\rel}$ in Lemma \ref{lem:same_gal_conj_count}. In addition, Lemma \ref{lem:same_gal_conj_count} shows that after fixing a Galois orbit of size $k$ among elements of the form $a_1 \pi_{L_1}^{b_1}$, there are $\frac{k}{\denom(\beta)}\gcd(p^{f_1}-1,e_{1,\rel})$ choices of $a_1,j_1$ such that $a_1 \pi_{L_1}^{b_1}$ belongs in the chosen Galois orbit.

Now, by Lemmas \ref{lem:nO_simple_case} and \ref{lem:same_gal_conj_count}, if $k \neq 1$, we get a total of $$\gcd(p^{f_2}-1,e_{2,\rel})\gcd(p^{f_1}-1,e_{1,\rel})\left(\frac{k}{\denom(\beta)}\right)^2P(f_{\base},k/\denom(\beta),p)$$ as our count for $$P((e_1^{f_1},e_2^{f_2}),\{\{1,2\}\},\{b_1,b_2\},\{k\},p).$$ Similarly, $$P((e_1^{f_1},e_2^{f_2}),\{\{1\},\{2\}\},\{b_1,b_2\},\{k,k\},p)$$ is $$\gcd(p^{f_1}-1,e_{1,\rel})\gcd(p^{f_2}-1,e_{2,\rel})\left(\frac{k}{\denom(\beta)}\right)^2(P(f_{\base},k/\denom(\beta),p)^2-P(f_{\base},k/\denom(\beta),p)).$$ One may interpret this as follows. For $$P((e_1^{f_1},e_2^{f_2}),\{\{1,2\}\},\{b_1,b_2\},\{k\},p),$$ we need to choose an orbit of size $k$ among elements of the form $a_i \pi_{L_j}^{b_i}$, and there are $P(f_{\base},k/\denom(\beta),p)$ many of such orbits, as mentioned in the above paragraph. Now, having chosen an orbit, we need to choose Galois conjugates $a_1 \pi_{L_{j_1}}^{\beta}$ and $a_2 \pi_{L_{j_2}}^{\beta}$ in that orbit, where $(a_1,j_1) \in T_{f_1,p} \times [\gcd(p^{f_1}-1,e_{1,\rel}]$ and $(a_2,j_2) \in T_{f_2,p} \times [\gcd(p^{f_2}-1,e_{2,\rel}]$. The number of ways we can choose $(a_i,j_i)$ is $\gcd(p^{f_i}-1,e_{i,\rel})\frac{k}{\denom(\beta)}$. Multiplying these together gives us the desired count. Similarly, for $P((e_1^{f_1},e_2^{f_2}),\{\{1,2\}\},\{b_1,b_2\},\{k\},p)$, we need to do the exact same thing, except we need to choose two different orbits of size $k$ (in order), and there are $$\frac{P(f_{\base},k/\denom(\beta),p)!}{(P(f_{\base},k/\denom(\beta),p)-2)!} = P(f_{\base},k/\denom(\beta),p)^2-P(f_{\base},k/\denom(\beta),p)$$ many choices.

We may apply the same reasoning to the general case of $\sigma = (e_1^{f_1}, \dots, e_m^{f_m})$. In this case, we obtain that if $k \neq 1$ and we have $M$ elements in our partition $\{E_l\}_{l=1}^M$, $$P(\sigma,\{E_l\}_{l=1}^M,\{b_i\}_{i=1}^m,\{k\}_{l=1}^M,p) = \left(\prod_{i=1}^m \gcd(p^{f_i}-1,e_{i,\rel})\right)\left(\frac{k}{\denom(\beta)}\right)^{m} \frac{P(f_{\base},k/\denom(\beta),p)!}{(P(f_{\base},k/\denom(\beta),p)-M)!}.$$


The calculation of $P(\sigma,\{E_l\}_{l=1}^M,\{b_i\}_{i=1}^m,\{1\}_{l=1}^M,p)$ is similar. In this case we also have to consider the cases of whether $[m] \neq I$, and so there is some $i$ so that $\beta < b_i/e_i$. Let $\{E_l\}_{l=1}^k$ be the partition we desire. Now if $\denom(\beta) \neq 1$, then actually there can only be one partition, as there's only one element of the form $a_{i,b_i}\pi_{L_{j_i}}^{b_i/e_{i,\rel}}$, where $a_{i,b_i} \in T_{f_i,p}$, that belongs to $K$, which is when $a_{i,b_i}=0$. Thus, in this case we obtain $\prod_{i=1}^m \gcd(p^{f_i}-1,e_{i,\rel})$ total choices, coming from the choices of $j_i$. Next, if $\denom(\beta)=1$, and so $\beta$ is an integer, then we get something more interesting. In this case, $\pi_{L_{j_i}}^{b_i/e_{i,\rel}} = \zeta_{p^{f_i}-1}^{j_i \beta}\pi_K^{\beta}$. Since we're considering elements of the form $a_{i,b_i}\pi_{L_{j_i}}^{b_i/e_{i,\rel}}$, and $\zeta_{p^{f_i}-1}^{j_i \beta} \in T_{f_i,p}$, we can ignore this factor by shifting each choice of $a_{i,b_i}$ by $\zeta_{p^{f_i}-1}^{j_i \beta}$. Thus, the choices of $j_i$ are independent from the choices of $a_{i,b_i}$. Hence, we just need to make choices for $a_{i,b_i}$. First, since we expect $a_{i,b_i} \pi_K^{\beta}$ to have size $1$ Galois orbit over $K$, it is actually in $K$, and so we must have that $a_{i, b_i} \in T_{f_{\base},p}$. Second, we need that $i_1$ and $i_2$ belong to the same partition if and only if $a_{i_1,b_{i_1}}=a_{i_2,b_{i_2}}$. Thus, if $I=[m]$, then we get a total of $$ \frac{p^{f_{\base}}!}{(p^{f_{\base}}-M)!}$$ choices of $\{a_{i,b_i}\}_{i \in I}$. If $I \neq [m]$, we are forced to choose $0$ to represent one of our Galois orbits, and so we get $$\frac{(p^{f_{\base}}-1)!}{(p^{f_{\base}}-1-(M-1))!}$$ choices of $\{a_{i,b_i}\}_{i \in I}$. Thus, $$P(\sigma,\{E_l\}_{l=1}^M,\{b_i\}_{i=1}^m,\{1\}_{l=1}^M,p)=\prod_{i=1}^m \gcd(p^{f_i}-1,e_{i,\rel})\begin{cases}
1 & \denom(\beta) \neq 1 \\
\frac{p^{f_{\base}}!}{(p^{f_{\base}}-M)!} & \denom(\beta) = 1, I = [m] \\
\frac{(p^{f_{\base}}-1)!}{((p^{f_{\base}}-1)-(M-1))!} & \denom(\beta) = 1, I \neq [m]
\end{cases}.$$

Now with the definition $F_k = \{i \in [m]: \exists l, n_{O,l}=k, i \in E_l\}$ recall that we had $$P(\sigma,\{E_l\},\{b_i\},\{n_{O,l}\},p)=\prod_k P_k(\sigma,\{E_l\},\{b_i\},\{n_{O,l}\},p),$$ as $k$ varies over integers for which $F_k$ is nonempty and $$P_k(\sigma,\{E_l\},\{b_i\},\{n_{O,l}\},p) = P((e_i^{f_i})_{i \in F_k}, \{E_l\}_{l, n_{O,l}=k}, \{b_i\}_{i \in F_k}, \{k\}_{l, n_{O,l}=k},p),$$ the RHS of which we just calculated. Thus, we obtain the following theorem, which is the main theorem of this section.
\begin{theorem}
\label{thm:aiji_count}
$P(\sigma,\{E_l\},\{b_i\},\{n_{O,l}\},p)=0$ unless all of the following are satisfied: \begin{enumerate}
    \item For all $l$ such that $n_{O,l} \neq 1$, $\denom(\beta)|n_{O,l}$.
    \item For all $i,l$ such that $i \in E_l$ and $n_{O,l} \neq 1$, $\frac{n_{O,l}}{\denom(\beta)} | f_{i,\rel}$.
    \item If $\denom(\beta) \neq 1$, then there is at most one $l_0$ so that $n_{O,l_0}=1$.
    \item If $I \neq [m]$, then there is some $l_0$, such that $E_{l_0}$ contains all of $[m] \setminus I$. 
\end{enumerate}
If all of the above conditions are satisfied, let $S$ be the set of $k$ for which there is some $l$ such that $n_{O,l}=k$. Also, let $F_k = \{i: i \in E_l, n_{O,l}=k\}$ be the set of all indices $i$ such that $i$ is in a partition $E_l$ with $n_{O,l}=k$. Then, $$P(\sigma,\{E_l\},\{b_i\},\{n_{O,l}\},p)=\prod_{k \in S} N_k,$$ where for $k \neq 1$, $$N_k=\left(\frac{k}{\denom(\beta)}\right)^{\# F_k} \frac{P(f_{\base},k/\denom(\beta),p)!}{(P(f_{\base},k/\denom(\beta),p)-\# \{l,n_{O,l}=k\})!},$$ and $$N_1=\begin{cases}
1 & \denom(\beta) \neq 1 \\
\frac{p^{f_{\base}}!}{(p^{f_{\base}}-\{l,n_{O,l}=1\})!} & \denom(\beta) = 1, I = [m] \\
\frac{(p^{f_{\base}}-1)!}{((p^{f_{\base}}-1)-(\{l,n_{O,l}=1\}-1))!} & \denom(\beta) = 1, I \neq [m]
\end{cases}.$$
\end{theorem} Note that it seems that we are missing one condition, which was that for all $k$, $\# \{l: n_{O,l}=k\} \leq P(f_{\base},k/\denom(\beta),p)$. But this condition is subsumed by the fact that if this condition is not satisfied, then $$\frac{P(f_{\base},k/\denom(\beta),p)!}{(P(f_{\base},k/\denom(\beta),p)-\# \{l,n_{O,l}=k\})!},$$ \textit{as a polynomial}, are $0$ when evaluated at $p$.

We end this section with the proof of Lemmas \ref{lem:nO_simple_case} and \ref{lem:same_gal_conj_count}. First we have the following short lemma
\begin{lemma}
\label{lem:equidistribution}
Fix $f$ and $e$ such that $p \nmid e$. Then, $x\frac{e}{\gcd(b,e)}+y\frac{b}{\gcd(b,e)} \mod (p^f-1)$ as $x$ ranges in $[p^f-1]$ and as $y$ ranges in $[\gcd(p^f-1,e)]$ is equidistributed among residue classes of $p^f-1$, with each class having $\gcd(p^f-1,e)$ elements.
\end{lemma}

\begin{proof}

This comes down to seeing that $\gcd(p^f-1,\frac{e}{\gcd(b,e)})|\gcd(p^f-1,e)$.


\end{proof}

\begin{proof} [Proof of Lemma \ref{lem:nO_simple_case}]
Recall $\pi_{L_{j_1}}=(\zeta_{p^{f_1}-1}^{j_1} \pi_K)^{1/e_{1,\rel}}$. Also, write $a_{1,b_1} = \zeta_{p^{f_1}-1}^{s}$. Then, $n_O(a_{1,b_1}\pi_{L_{j_1}}^{b_1})$ is the number of Galois conjugates of $ \zeta_{p^{f_1}-1}^{s}(\zeta_{p^{f_1}-1}^{j b_1 /e_{1,\rel}} \pi_K)^{b_1 /e_{1,\rel}}$. The set of these conjugates is precisely $$\{(\zeta_{p^{f_1}-1}^{s}\zeta_{p^{f_1}-1}^{jb_1 /e_{1,\rel}})^{p^{f_{\base}r_1}} \zeta_{e_{1,\rel}}^{b_1r_2}\pi_K^{b_1 /e_{1,\rel}}:r_1 \in \bbZ/(f_1/f_{\base})\bbZ, r_2 \in \bbZ/(e_{1,\rel})\bbZ\}.$$ So, we wish to find the number of distinct elements of this set. We do this by turning to modular arithmetic by looking at the exponents. That is, we have a bijection from $$\{p^{f_{\base}r_1}(se_{1,\rel}+jb_1) +b_1r_2(p^{f_1}-1) \mod (p^{f_1}-1)e_{1,\rel}:r_1 \in \bbZ/(f_1/f_{\base})\bbZ, r_2 \in \bbZ/(e_{1,\rel})\bbZ\}$$ to $$\{(\zeta_{p^{f_1}-1}^{s}\zeta_{p^{f_1}-1}^{jb_1 /e_{1,\rel}})^{p^{f_{\base}r_1}} \zeta_{e_{1,\rel}}^{b_1r_2}\pi_K^{b_1 /e_{1,\rel}}:r_1 \in \bbZ/(f_1/f_{\base})\bbZ, r_2 \in \bbZ/(e_{1,\rel})\bbZ\}$$ sending $x \mod (p^{f_1}-1)e_{1,\rel}$ to $\zeta_{(p^{f_1}-1)e_{1,\rel}}^x \pi_K^{b_1/e_{1,\rel}}$. So, equivalently, we wish to find the number of distinct elements of $$\{p^{f_{\base}r_1}(se_{1,\rel}+jb_1) +b_1r_2(p^{f_1}-1) \mod (p^{f_1}-1)e_{1,\rel}:r_1 \in \bbZ/(f_1/f_{\base})\bbZ, r_2 \in \bbZ/(e_{1,\rel})\bbZ\}.$$ Taking these elements modulo $b_1(p^{f_1}-1)$, we see that the number of such distinct elements is then equal to $\frac{e_{1,\rel}}{\gcd(b_1,e_{1,\rel})}$ times the number of distinct elements of $$S_{s,j_1}=\{p^{f_{\base}r_1}(se_{1,\rel}+jb_1) \mod (p^{f_1}-1)\gcd(b_1,e_{1,\rel}):r_1 \in \bbZ/(f_1/f_{\base})\bbZ\}.$$ Now the number of distinct elements to the above equation is the smallest $r_1$ such that \begin{equation}
    (p^{r_1}-1)(se_{1,\rel}+j_1b_1) \equiv 0 \mod (p^{f_1}-1)\gcd(b_1,e_{1,\rel}). \label{eqn:temp2}
\end{equation} We will denote the LHS of this congruence to be $E(r_1,s,j_1)$. One can see by minimality of such $r_1$ that $r_1|f_1/f_{\base}$. Thus, so far, what we have is that $$n_O(\zeta_{p^{f_1}-1}^{s}(\zeta_{p^{f_1}-1}^{jb_1/e_{1,\rel}}\pi_K)^{b_1/e_{1,\rel}}) = \frac{e_{1,\rel}}{\gcd(b_1,e_{1,\rel})} \min(\{r_1':r_1', E(r_1',s,j_1) \equiv 0 \mod (p^{f_1}-1)\gcd(b_1,e_{1,\rel})\}).$$ However, what we want to find is the number of choices $(s,j_1)$ that give a certain number of Galois conjugates. So, we want to fix $r_1$, and find the number of choices of $(s,j_1)$ such that $$r_1=\min(\{r_1':r_1',E(r_1',s,j_1) \equiv 0 \mod (p^{f_1}-1)\gcd(b_1,e_{1,\rel})\}).$$ These choices will all give $$\frac{r_1 e_{1,\rel}}{\gcd(b_1,e_{1,\rel})}=r_1\denom(\beta)$$ conjugates.

First, we will find the number of solutions $s,j_1$ to $$E(r_1,s,j_1) \equiv 0 \mod (p^{f_1}-1)\gcd(b_1,e_{1,\rel}),$$ after fixing any $r_1$ such that $r_1|f_1/f_{\base}$. These solutions $(s,j_1)$ are precisely the ones such that $$\min(\{r_1':r_1',E(r_1',s,j_1) \equiv 0 \mod (p^{f_1}-1)\gcd(b_1,e_{1,\rel})\})$$ \textit{divide} $r_1$, which is not exactly what we want, but it will be after applying inclusion exclusion, as we will see. So, for now, we will fix an $r_1|f_1/f_{\base}$ and find the number of solutions $s,j_1$ to Equation \ref{eqn:temp2}. Dividing both sides by $(p^{r_1}-1)\gcd(b_1,e_{1,\rel})$, our equation turns into $$s\frac{e_{1,\rel}}{\gcd(b_1,e_{1,\rel}))}+j_1\frac{b_1}{\gcd(b_1,e_{1,\rel})} \equiv 0 \mod \frac{p^{f_1}-1}{p^{r_1}-1}.$$ Now we apply Lemma \ref{lem:equidistribution}, obtain that there are $$(p^{r_1}-1)\gcd(p^{f_1}-1,e_{1,\rel})$$ solutions. With this, we may apply inclusion exclusion principle, and obtain that the number of choices of $(s,j_1)$ such that $$r_1=\min(\{r_1':r_1',E(r_1',s,j_1) \equiv 0 \mod (p^{f_1}-1)\gcd(b_1,e_{1,\rel})\})$$ is \begin{align*}
    &=\gcd(p^{f_1}-1,e_{1,\rel})\big(p^{r_1}-1-\sum_{q_1|r_1}(p^{f_{\base}\frac{r_1}{q_1}}-1)\\
    &+\sum_{q_1 \neq q_2, q_i | r_1}(p^{f_{\base}\frac{r_1}{q_1q_2}}-1)- \dots \big)\\
    &=\gcd(p^{f_1}-1,e_{1,\rel})r_1P(f_{\base},r_1,p).
\end{align*} In other words, the number of choices of $(s,j_1)$ such that $\zeta_{p^{f_1}-1}^{s}(\zeta_{p^{f_1}-1}^{j b_1 /e_{1,\rel}} \pi_K)^{b_1 /e_{1,\rel}}$ has $k$ Galois conjugates is $\gcd(p^{f_1}-1,e_{1,\rel})\frac{k}{\denom(\beta)}P(f_{\base},\frac{k}{\denom(\beta)},p)$.

\end{proof}

\begin{proof}[proof of Lemma \ref{lem:same_gal_conj_count}]
We again convert the problem into modular arithmetic. Write $a_{i,b_i}=\zeta_{p^{f_i}-1}^{s_i}$ and $a_{2,b_2}=\zeta_{p^{f_2}-1}^{s_2}$, where $s_i \in \bbZ/(p^{f_i}-1)\bbZ$. Then, as in the proof of Lemma \ref{lem:nO_simple_case}, we had that $$\{p^{f_{\base}r_1}(se_{1,\rel}+jb_1) +b_1r_2(p^{f_1}-1) \mod (p^{f_1}-1)e_{1,\rel}:r_1 \in \bbZ/(f_1/f_{\base})\bbZ, r_2 \in \bbZ/(e_{1,\rel})\bbZ\}$$ corresponds to the Galois conjugates of $a_{1,b_1}\pi_{L_{j_1}}^{b_1}$, via the map that sends $x \mod (p^{f_1}-1)e_{1,\rel}$ to $\zeta_{(p^{f_1}-1)e_{1,\rel}}^x \pi_K^{b_1/e_{1,\rel}}$. Since we need to consider $a_{2,b_2} \pi_{L_{j_2}}^{b_2}$, we need to actually express our conjugates in terms of roots of unity with order $(p^{f_1}-1)e_{1,\rel}(p^{f_2}-1)e_{2,\rel}$. Thus, for the sake of this proof, the Galois conjugates of $a_{1,b_1}\pi_{L_{j_1}}^{b_1}$ will correspond to $$\{(p^{f_2}-1)e_{2,\rel}(p^{f_{\base}r_1}(se_{1,\rel}+jb_1) +b_1r_2(p^{f_1}-1)) \mod (p^{f_1}-1)e_{1,\rel}(p^{f_2}-1)e_{2,\rel}:$$$$r_1 \in \bbZ/(f_1/f_{\base})\bbZ, r_2 \in \bbZ/e_{1,\rel}\bbZ\}.$$ Similarly, elements $a_{2,b_2}\pi_{L_{j_2}}^{b_2}$ will correspond to $$(p^{f_1}-1)e_{1,\rel}(s_2e_{2,\rel}+j_2b_2) \mod (p^{f_1}-1)e_{1,\rel}(p^{f_2}-1)e_{2,\rel}.$$ Thus, in terms of modular arithmetic, we must find $$(s_2,j_2)\in \bbZ/(p^{f_2}-1)\bbZ \times [\gcd(p^{f_2}-1,e_{2,\rel})]$$ so that there is $r_1,r_2$ such that $$(p^{f_2}-1)e_{2,\rel}(p^{f_{\base}r_1}(se_{1,\rel}+jb_1) +b_1r_2(p^{f_1}-1)) \equiv (p^{f_1}-1)e_{1,\rel}(s_2e_{2,\rel}+j_2b_2) $$$$ \mod (p^{f_1}-1)e_{1,\rel}(p^{f_2}-1)e_{2,\rel}.$$ Fixing $r_1$, we have a solution $r_2$ to the above equation if and only if $$p^{r_1 f_{\base}}s_1e_{1,\rel}(p^{f_2}-1)e_{2,\rel}+p^{r_1f_{\base}}j_1b_1(p^{f_2}-1)e_{2,\rel} $$$$ \equiv s_2e_{2,\rel}(p^{f_1}-1)e_{1,\rel}+j_2b_2(p^{f_1}-1)e_{1,\rel} \mod (p^{f_1}-1)\gcd(b_1,e_{1,\rel})(p^{f_2}-1)e_{2,\rel}.$$ Thus, this is the equation that we need to solve. Now by Equation \ref{eqn:temp2}, we have $$\frac{(p^{f_1}-1)\gcd(b_1,e_{1,\rel})}{p^{f_{\base}k}-1}|(s_1e_{1,\rel}+j_1b_1).$$ Thus, we may write $s_1e_{1,\rel}+j_1b_1=y\frac{(p^{f_1}-1)\gcd(b_1,e_{1,\rel})}{p^{f_{\base}k}-1}$ for some integer $y$. Thus, our congruence becomes $$\frac{yp^{r_1 f_{\base}}(p^{f_2}-1)e_{2,\rel}}{p^{f_{\base}k}-1} \equiv s_2e_{2,\rel}\frac{e_{1,\rel}}{\gcd(b_1,e_{1,\rel})}+j_2b_2\frac{e_{1,\rel}}{\gcd(b_1,e_{1,\rel})} \mod (p^{f_2}-1)e_{2,\rel}.$$ Now as we had in the statement of the lemma, $\frac{e_{1,\rel}}{\gcd(b_1,e_{1,\rel})}=\frac{e_{2,\rel}}{\gcd(b_2,e_{2,\rel})}$. Thus, we may divide by $e_{2,\rel}$ on both sides and obtain $$\frac{yp^{r_1 f_{\base}}(p^{f_2}-1)}{p^{f_{\base}k}-1} \equiv s_2\frac{e_{2,\rel}}{\gcd(b_2,e_{2,\rel})}+j_2\frac{b_2}{\gcd(b_2,e_{2,\rel})} \mod (p^{f_2}-1).$$ Now, by Lemma \ref{lem:equidistribution}, for each choice of $r_1$, there are $\gcd(p^{f_2}-1,e_{2,\rel})$ choices of $s_2,j_2$ satisfying the above congruence. Thus, these choices $s_2,j_2$ are precisely the ones such that $\zeta_{p^{f_2}-1}^{s_2}\pi_{L_{j_2}}^{b_2}$ is a Galois conjugate of $\zeta_{p^{f_1}-1}^{s_1}\pi_{L_{j_1}}^{b_1}$. Furthermore, for $r_1 \in [\frac{k}{\denom(\beta)}]$, we get different elements on the LHS. Otherwise, we would have that $$(p^{r_1'f_{\base}}-1)(s_2\frac{e_{2,\rel}}{\gcd(b_2,e_{2,\rel})}+j_2\frac{b_2}{\gcd(b_2,e_{2,\rel})}) \equiv 0 \mod p^{f_2}-1,$$ for some $r_1' < r_1$, which would imply that our choice of $s_2, j_2$ gives us that $\zeta_{p^{f_2}-1}^{s_2}\pi_{L_{j_2}}^{b_2}$ has less than $k=r_1 \denom(\beta)$ Galois conjugates, which would be a contradiction to the fact that these are Galois conjugates of $\zeta_{p^{f_1}-1}^{s_1}\pi_{L_{j_1}}^{b_1}$. Thus, we get a total of $$\frac{k}{\denom(\beta)}\gcd(p^{f_2}-1,e_{2,\rel})$$ choices of $s_2,j_2$
\end{proof}

\section{Large $q$ limit of $\rho$}

In this section, we study the large $q$ limit of $\rho(\sigma,e,f;p)$ and prove Theorem \ref{thm:large_q_limit_rho}. By large $q$ limit, we mean asymptotics as $p$ or $f$ (or both) approaches infinity. For the rest of this section, we will use absolute notation, as we did for section \ref{sec:generating_function} and so in particular we have that $e_{\base}|e_i$ and $f_{\base}|f_i$ for all $e_i,f_i$ in $\sigma$. 

We begin with a formula for $\rho$ in terms of $G$. Combining Equation \ref{eqn:rho} and Lemma \ref{lem:cond_disc_order}, we may rewrite $\rho$ as $$\rho(\sigma,e_{\base},f_{\base};p)=\frac{w(d)\sum_{A \subset [m]}N_{1,A}N_{2,A}}{\perm(\sigma)q^{\sum_{i=1}^m (e_{i,\rel}-1)f_{i,\rel}/2}\prod_{i=1}^m f_{i,\rel}},$$ where $$N_{1,A}=\frac{\sum_{\substack{L_A = \prod_{i \in A} L_{j_i}\\ j_i \in [\gcd(q^{f_{i,\rel}}-1,e_{i,\rel})]}}\int_{\alpha \in \cO_{L_A}} q^{-\frac{1}{2}v_K(\Delta_{P_{\alpha/K}})} d\lambda_{L_A}}{\prod_{i \in A}\gcd(q^{f_{i,\rel}}-1,e_{i,\rel})}$$ and $$N_{2,A}=\frac{\sum_{\substack{L_{A^c} = \prod_{i \notin A} L_{j_i}\\ j_i \in [\gcd(q^{f_{i,\rel}}-1,e_{i,\rel})]}}\int_{\alpha \in m_{L_{A^c}}} q^{-\frac{1}{2}v_K(\Delta_{P_{\alpha/K}})} d\lambda_{L_{A^c}}}{\prod_{i \notin A} \gcd(q^{f_{i,\rel}}-1,e_{i,\rel})}.$$ Here $q=p^{f_{\base}}$ is the size of the residue field of $K$. We recognize that $N_{1,A}$ is just $G(\sigma_A,e_{\base},f_{\base},\{0\}_{i \in A},p)(q^{-e_{\base}/2})$, where $\sigma_A = (e_i^{f_i})_{i \in A}$, and similarly for $N_{2,A}$. Writing $$G(A,0)(q,t)=G(\sigma_A,e_{\base},f_{\base},\{0\}_{i \in A},p)(t)$$ and $$G(A,1)(q,t)=G(\sigma_A,e_{\base},f_{\base},\{1\}_{i \in A},p)(t),$$ where $q=p^{f_{\base}}$, we obtain the formula \begin{equation} \label{eqn:rho2}
    \rho(\sigma,e_{\base},f_{\base};p)=\frac{w(d)\sum_{A \subset [m]}G(A,0)(q^{-e_{\base}/2})G(A^c,1)(q^{-e_{\base}/2})}{\perm(\sigma)q^{\sum_{i=1}^m (e_{i,\rel}-1)f_{i,\rel}/2}\prod_{i=1}^m f_{i,\rel}},
\end{equation} which we may use to compute $\rho$ from $G$ (here we define $G(A,1)=1$ if $A$ is empty.

Thus, if suffices to compute asymptotics of $G$. Let $L/K$ be an \'etale extension with the appropriate splitting type. Our approach revolves around the following observation: Most Galois conjugates of elements in $\cO_L$ try to be as far as possible from one another. In other words, most elements $\alpha$ in $\cO_L$ have the least possible $v_p(\Delta_{P_{\alpha}/K})$. Thus, what we will show is: \begin{enumerate}
    \item Find the smallest possible $v_p(\Delta_{P_{\alpha}/K})$ in $\cO_L$.
    \item Show that as $q$ tends to infinity, the probability that an element $\alpha \in \cO_L$ attains this minimal valuation approaches $1$.
\end{enumerate} Now recall our definition of $G(\sigma,e_{\base},f_{\base},\{b_i\})$ in terms of coefficients $a(\sigma,e_{\base},f_{\base},c,\{b_i\}_{i=1}^m,p)$. The above list translates to \begin{enumerate}
    \item Find the smallest $c$ such that $a(\sigma,e_{\base},f_{\base},c,\{b_i\}_{i=1}^m,p)$ is nonzero. Call this $c_0$.
    \item Show $\lim_{q \to \infty} a(\sigma,e_{\base},f_{\base},c_0,\{b_i\}_{i=1}^m,p)=1$.
\end{enumerate} The combination of these two facts then gives us large $q$ limit asymptotics for $$G(\sigma,e_{\base},f_{\base},\{0\}_{i=1}^m,p)(q^{-e_{\base}/2}).$$ We will then give a bound on the size of $$G(\sigma,e_{\base},f_{\base},\{1\}_{i=1}^m,p)(q^{-e_{\base}/2}).$$ It turns out that these facts will suffice to calculate the asymptotics of $\rho$.

Now we start with our first step. We need to find the minimal $c$ so that $a(\sigma,e_{\base},f_{\base},c,\{0\}_{i=1}^m,p)$ is nonzero. Fix a base field $K$ with $e(K/\bbQ_p)=e_{\base}$ and $f(K/\bbQ_p)=f_{\base}$. Then, $a(\sigma,e_{\base},f_{\base},c,\{0\}_{i=1}^m,p)$ is nonzero if and only if there is some element $(\alpha_i)$ generating an extension of $K$ with splitting type $\sigma$ such that $v_p(\Delta_{P_{(\alpha_i)}/K})=c$. 

\begin{lemma}\label{lem:smallest_disc_element}
The smallest $c$ such that there is some element $(\alpha_i)$ generating an extension of $K$ with splitting type $\sigma$ such that $v_p(\Delta_{P_{(\alpha_i)}/K})=c$ is $$c_0:=\sum_{i=1}^m\frac{1}{e_{\base}}f_{i,\rel}(e_{i,\rel}-1).$$ Moreover, the elements $(\alpha_i)$ that achieve this are precisely the ones that satisfy
\begin{enumerate}
    \item For all $i$, the first term in the Teichm\"uller expansion of $\alpha_i$ are primitive $p^{f_i}-1$ roots of unity, and these roots of unity belong in different Galois orbits over $K$ as $i$ varies.
    \item For all $i$, the second term in the Teichm\"uller expansion of $\alpha_i$ is nonzero.
\end{enumerate}
\end{lemma}
\begin{proof}


We approach this problem by essentially following the recursive steps as in Proposition \ref{prop:induction_a_with_P}. We will prove lower bounds on the valuation of discriminant of a minimal polynomial generating an extension with appropriate splitting type, and see along the way that these bounds are achieved.

We have that (here we use the same notation as Section \ref{sec:generating_function}) $$\Delta_{P_{(\alpha_i)}/K} = \prod_{(i,k) \neq (i',k'), k,k' \in [d_i]} (\alpha_{i,k}-\alpha_{i',k'}),$$ where $\alpha_{i,k}$ are Galois conjugates of $\alpha_i$ over $K$. So, we wish to maximize distances between Galois conjugates. Now we use equation \ref{eqn:disc_recursion} (and its associated notations) and try to minimize distances in each of the terms. So, we have two cases: $(i,k) \sim (i',k')$ and $(i,k) \not \sim (i',k')$. In the first case, where $(i,k) \sim (i',k')$, we have the trivial bound that $$v_p(\alpha_{i,k}-\alpha_{i',k'}) \geq 0.$$ In the second case, where $(i,k) \sim (i',k')$, we have that $\alpha_{i,k}$ and $\alpha_{i',k'}$ share a common lowest order term $\gamma$ (in terms of the Teichm\"uller expansion). Thus, $$v_p(\alpha_{i,k}-\alpha_{i',k'}) = v_p((\alpha_{i,k}-\gamma)-(\alpha_{i',k'}-\gamma)) \geq \min(v_p(\alpha_{i,k}-\gamma_{i,k}),v_p(\alpha_{i',k'}-\gamma_{i,k}))$$ by definition of a valuation. Now $\alpha_{i,k}-\gamma$ and $\alpha_{i',k'}-\gamma$ no longer have valuation $0$, and since they belong in extensions with ramification indices $e_i$ and $e_{i'}$, respectively, we have that $v_p(\alpha_{i,k}-\gamma) \geq \frac{1}{e_i}$ and $v_p(\alpha_{i',k'}-\gamma) \geq \frac{1}{e_{i'}}$. Thus, we obtain the lower bound that $$v_p(\alpha_{i,k}-\alpha_{i',k'}) \geq \min(\frac{1}{e_i},\frac{1}{e_{i'}}).$$ We may compactly write these bounds as $$v_p(\alpha_{i,k}-\alpha_{i',k'}) \geq \begin{cases}
0 & \alpha_{i,k},\alpha_{i',k'}\text{ have different lowest order term Galois orbits} \\
\min(\frac{1}{e_i},\frac{1}{e_{i'}}) & \text{ otherwise}
\end{cases}$$



With these lower bounds, we see that in order to minimize the valuation of the discriminant, we would like every pair $(i,k) \neq (i',k')$ to satisfy that $v_p(\alpha_{i,k}-\alpha_{i',k'})=0$. In other words, we want $\alpha_{i,k}$ and $\alpha_{i',k'}$ to be have valuation $0$ and also have different lowest order terms in the Teichm\"uller expansion. This is certainly possible if $i \neq i'$ (granted $q$ is large enough). Indeed for any choice $(\alpha_i)$ generating an extension $\prod_{i=1}^m L_i$ with the appropriate splitting, we can adjust the valuation-$0$ terms of the Teichm\"uller expansion of $\alpha_i$ in terms of $\pi_{L_i}$ in order to make sure that $v_p(\alpha_{i,k}-\alpha_{i',k'})=0$ for $i \neq i'$ (we can do this by choosing appropriate roots of unity belonging to different Galois orbits over $K$). Thus, the problem reduces down to the case of field extensions. That is, given two numbers $e_{abs}$ and $f_{abs}$, we ask for the smallest valuation possible of $\Delta_{P_{\alpha}/K}$, given that $\alpha$ generates an extension $L/K$ of ramification index $e(L/\bbQ_p)=e_{abs}$ and inertia degree $f(L/\bbQ_p)=f_{abs}$.

We now focus on the field case. Fix an extension $L/K$ with $e(L/\bbQ_p)=e_{abs}$ and $f(L/\bbQ_p)=f_{abs}$. Define $e_{\rel}=e_{abs}/e_{\base}$ and $f_{\rel}=f_{abs}/f_{\base}$. Take any $\alpha \in L$. Write it as $a_0 + \delta$, where $a_0\in T_{f_{abs}}$ and $\delta \in \fkm_L$. That is, we are extracting the first coefficient of its Teichm\"uller expansion. Now let $n_O(a_0)$ be the number of conjugates of $a_0$, and so $n_O(a_0)|f_{\rel}$. Then, by Equation \ref{eqn:disc_recursion} and the bound above, we have that $$v_p(\Delta_{\alpha}) \geq \frac{1}{e_{abs}} n_O(a_0) (d/n_O(a_0))(d/n_O(a_0)-1)$$ (here $d=e_{\rel}f_{\rel}$). Thus, in order to minimize $v_p(\Delta_{\alpha})$, we must choose $a_0$ so that $n_O(a_0)$ is as large as possible, and so we should take $a_0$ to be a primitive root of unity of order $q^{f_{\rel}}-1$, in which case $n_O(a_0)=f_{\rel}$. With this choice, we have that $$v_p(\Delta_{\alpha}) \geq \frac{1}{e_{abs}}f_{\rel}e_{\rel}(e_{\rel}-1).$$ Now $\delta$ generates $L$ over $K[a_0]$ as a totally ramified extension with relative ramification index $e_{\rel}$. Denote $\{\delta_i\}_{i=1}^{e_{\rel}}$ to be the Galois conjugates of $\delta$. Then, the lower bound $\frac{1}{e_{abs}}f_{\rel}e_{\rel}(e_{\rel}-1)$ will only be achieved if $v_p(\delta_i-\delta_j) = \frac{1}{e_{abs}}$, for all $i \neq j$. This happens if and only if $\delta_i \notin \fkm_L^2$, which gives us the second condition.
\end{proof}


Thus, at this point, we know that $G(\sigma_A,e_{\base},f_{\base},\{0\}_{i \in A})(t)$ is of the form $$a(\sigma,e_{\base},f_{\base},c_0,\{0\})t^{c_0}+H.O.T.,$$ We want to find large $q$ limit when we evaluate $G$ at $t=q^{-e_{\base}/2}$. Now, if we evaluate $G$ at $t=1$, we obtain the sum of the coefficients $a(\sigma,e_{\base},f_{\base},c,\{0\}_{i=1}^m,p)$, which by definition is $$\frac{1}{\prod_{i \in [m]} \gcd(p^{f_i}-1,e_{i,\rel})}\sum_{j_i \in [\gcd(p^{f_i}-1,e_{i,\rel})]} a(\{L_{j_i}\},K,c,\{0\}_{i=1}^m).$$ Now summing this over $c$, we get $$\sum_c\frac{1}{\prod_{i \in [m]} \gcd(p^{f_i}-1,e_{i,\rel})}\sum_{j_i \in [\gcd(p^{f_i}-1,e_{i,\rel})]} a(\{L_{j_i}\},K,c,\{0\}_{i=1}^m)$$$$=\frac{1}{\prod_{i \in [m]} \gcd(p^{f_i}-1,e_{i,\rel})}\sum_{j_i \in [\gcd(p^{f_i}-1,e_{i,\rel})]} \sum_c a(\{L_{j_i}\},K,c,\{0\}_{i=1}^m)$$\begin{equation} \label{eqn:temp1131}
    =\frac{1}{\prod_{i \in [m]} \gcd(p^{f_i}-1,e_{i,\rel})}\sum_{j_i \in [\gcd(p^{f_i}-1,e_{i,\rel})]} 1 = 1.
\end{equation} Thus, if we can show that $\lim_{q \to \infty} a(\sigma,e_{\base},f_{\base},c_0,\{0\}_{i=1}^m,p) = 1$, then the large $q$ limit of $G(\sigma,e_{\base},f_{\base},\{0\}_{i \in A})(q^{-e_{\base}/2})$ is precisely the large $q$ limit of $q^{-c_0e_{\base}/2}$. This is the second step.

Thus, we now aim to show that $\lim_{q \to \infty} a(\sigma,e_{\base},f_{\base},c_0,\{0\}_{i=1}^m,p) = 1$. This purely comes down to following the appropriate recursive branch (the one associated to algebraic numbers over $K$ satisfying the two conditions in Lemma \ref{lem:smallest_disc_element}) in proposition \ref{prop:induction_a_with_P}. First, we get a large $q$ limit asymptotic for $P(\sigma, \{E_l\}, \{b_i\}, \{n_{O,l}\},p)$. We do this by using Theorem \ref{thm:aiji_count}. We also follow the notations in the theorem.

\begin{cor} \label{cor:large_q_limit_P}
$P(\sigma,\{E_l\},\{b_i\},\{n_{O,l}\},p)=0$ unless all of the following are satisfied: \begin{enumerate}
    \item For all $l$ such that $n_{O,l} \neq 1$, $\denom(\beta)|n_{O,l}$.
    \item For all $i,l$ such that $i \in E_l$ and $n_{O,l} \neq 1$, $\frac{n_{O,l}}{\denom(\beta)} | f_{i,\rel}$.
    \item If $\denom(\beta) \neq 1$, then there is at most one $l_0$ so that $n_{O,l_0}=1$.
    \item If $I \neq [m]$, then there is some $l_0$, such that $E_{l_0}$ contains all of $[m] \setminus I$. 
\end{enumerate}
If all of the above conditions are satisfied, let $S$ be the set of $k$ for which there is some $l$ such that $n_{O,l}=k$. Also, let $F_k = \{i: i \in E_l, n_{O,l}=k\}$ be the set of all indices $i$ such that $i$ is in a partition $E_l$ with $n_{O,l}=k$. Then, $$P(\sigma,\{E_l\},\{b_i\},\{n_{O,l}\},p) \sim \prod_{k \in S} N_k',$$ where for $k \neq 1$ $$N_k'=\left(\frac{k}{\denom(\beta)}\right)^{\# F_k} \left(\frac{p^{f_{\base} k/\denom(\beta)}}{k/\denom(\beta)}\right)^{\#\{l,n_{O,l}=k\}}$$ and $$N_1'=\begin{cases}
1 & \denom(\beta) \neq 1 \\
q^{\#\{l,n_{O,l}=1\}} & \denom(\beta) = 1, I = [m] \\
q^{\#\{l,n_{O,l}=1\}-1} & \denom(\beta) = 1, I \neq [m]
\end{cases}$$
\end{cor}
\begin{proof}
We first start with the large $q$ limit asymptotic of $P(f_{\base},k',p)$. Following its definition, this is $P(f_{\base},k',p) \sim \frac{p^{f_{\base} k'}}{k'}=\frac{q^{k'}}{k'}$. Next, we use Theorem \ref{thm:aiji_count}. By definition of $N_k$, we have that for $k \neq 1$, $$N_k \sim \left(\frac{k}{\denom(\beta)}\right)^{\# F_k} \left(\frac{q^{ k/\denom(\beta)}}{k/\denom(\beta)}\right)^{\#\{l,n_{O,l}=k\}}.$$ We also have $$N_1 \sim \begin{cases}
1 & \denom(\beta) \neq 1 \\
q^{\#\{l,n_{O,l}=1\}} & \denom(\beta) = 1, I = [m] \\
q^{\#\{l,n_{O,l}=1\}-1} & \denom(\beta) = 1, I \neq [m]
\end{cases}.$$
\end{proof}

\begin{lemma}
$\lim_{q \to \infty} a(\sigma,e_{\base},f_{\base},c_0,\{0\}_{i=1}^m,p) = 1$.
\end{lemma}
\begin{proof}
We follow the recursion in Proposition \ref{prop:induction_a_with_P}, which expresses $a(\sigma,e_{\base},f_{\base},c_0,\{0\}_{i=1}^m,p)$ as a sum. Now, by condition 1 in Lemma \ref{lem:smallest_disc_element}, the only nonzero term in this sum comes from the partition $\{E_i\}_{i=1}^m$ of $[m]$ consist of singleton sets with $E_i=\{i\}$, $n_{O,i}=f_{i,\rel}$, and $c_i = \frac{1}{e_{\base}}(e_{i,\rel}-1)$. Thus, $$a(\sigma,e_{\base},f_{\base},c_0,\{0\}_{i=1}^m,p) = P(\sigma, \{E_i\}, \{0\}, \{n_{O,i}\},p) \prod_i a((e_i^{f_i}),e_{\base},f_i,\frac{1}{e_{\base}}(e_{i,\rel}-1),1).$$ By Corollary \ref{cor:large_q_limit_P}, $$P(\sigma, \{E_i\}, \{0\}, \{n_{O,i}\},p) \sim \prod_i q^{f_{i,\rel}}.$$ So, it remains to find $$a((e_i^{f_i}),e_{\base},f_i,\frac{1}{e_{\base}}(e_{i,\rel}-1),1),$$ to which we will again apply same reasoning as before, except this time we apply condition 2 in Lemma \ref{lem:smallest_disc_element}. This will yield that $$a((e_i^{f_i}),e_{\base},f_i,\frac{1}{e_{\base}}(e_{i,\rel}-1),1) \sim q^{-f_{i,\rel}}.$$ Thus, we obtain that $a(\sigma,e_{\base},f_{\base},c_0,\{0\}_{i=1}^m,p) \sim 1$, as desired.
\end{proof}

Finally, we obtain our desired asymptotic for $G$.

\begin{prop} \label{prop:large_q_limit_G}
$G(\sigma,e_{\base},f_{\base},\{0\}_{i=1}^m,p)(q^{-e_{\base}/2}) \sim q^{-\frac{1}{2}\sum_{i=1}^m f_{i,\rel}\frac{e_{i,\rel}-1}{2}}$
\end{prop}
\begin{proof}
The RHS is the expansion of $q^{-c_0e_{\base}/2}$, after replacing $c_0$ by its definition in Lemma \ref{lem:smallest_disc_element}.
\end{proof}

Now to proceed to computing an asymptotic for $\rho$, we need the following trivial bound.

\begin{lemma} \label{lem:large_q_limit_G1_bound}
Let $e_{min} = \min_i(e_{i,\rel})$. Then, $$G(\sigma,e_{\base},f_{\base},\{1\}_{i=1}^m,p)(q^{-e_{\base}/2}) = o(\sum_{i=1}^m f_{i,\rel}(e_{i,\rel}-1)).$$
\end{lemma}
\begin{proof}
First we find the smallest exponent of $t$ in $G(\sigma,e_{\base},f_{\base},\{1\}_{i=1}^m,p)(t)$. $\frac{1}{e_i}d_i(d_i-1)$ is the minimal valuation of the discriminant of a minimal polynomial of an element $\alpha_i$ in $\fkm_{L_i}$, since there are $d_i(d_i-1)$ choices of pairs of Galois conjugates, and for each pair $\delta,\gamma$ of such Galois conjugates, $v_p(\delta-\gamma) \geq \frac{1}{e_i}$. Now, $v_p(\Delta_{P_{(\alpha_i)}/K})$ takes substantially more into account, and we should also consider the contribution to the discriminant coming from differences of a Galois conjugates of $\alpha_i$ and a Galois conjugate of $\alpha_j$, where $i \neq j$. However, for the sake of an upper bound, we may ignore this. Thus, the smallest exponent of $t$ is at least $\sum_i \frac{1}{e_i}d_i(d_i-1)$. Next, as in Equation \ref{eqn:temp1131}, the sum of all coefficients of $G(\sigma,e_{\base},f_{\base},\{1\}_{i=1}^m,p)$ gives $\prod_{i=1}^m q^{-f_{i,\rel}}$. Thus, we obtain that \begin{align*}
    G(\sigma,e_{\base},f_{\base},\{1\}_{i=1}^m,p)(q^{-e_{\base}/2}) & \leq \prod_{i=1}^m q^{-f_{i,\rel}} \cdot (q^{-e_{\base}/2})^{\sum_i \frac{1}{e_i}d_i(d_i-1)}\\
    & = \prod_{i=1}^m q^{-f_{i,\rel}} \cdot q^{-\frac{1}{2}\sum_{i=1}^m \frac{1}{e_{i,\rel}}d_i(d_i-1)}.
\end{align*} Since $\sum_{i=1}^m f_{i,\rel}+\sum_{i=1}^m \frac{1}{e_{i,\rel}}d_i(d_i-1) > \sum_{i=1}^m f_{i,\rel}(e_{i,\rel}-1)$, for any choice of $e_i$ and $f_i$, we obtain our desired bound. 
\end{proof}

We are now ready to prove Theorem \ref{thm:large_q_limit_rho}.
\begin{proof} [Proof of theorem \ref{thm:large_q_limit_rho}]
We have by Equation \ref{eqn:rho2}$$\rho(\sigma,e_{\base},f_{\base};p)=\frac{w(d)\sum_{A \subset [m]}G(A,0)(q^{-e_{\base}/2})G(A^c,1)(q^{-e_{\base}/2})}{\perm(\sigma)q^{\sum_{i=1}^m (e_{i,\rel}-1)f_{i,\rel}/2}\prod_{i=1}^m f_{i,\rel}}.$$ Now by Proposition \ref{prop:large_q_limit_G} and Lemma \ref{lem:large_q_limit_G1_bound}, most terms in this sum drop out, except $$G([m],0)(q^{-e_{\base}/2}) \sim q^{-\frac{1}{2}\sum_{i=1}^m f_{i,\rel}(e_{i,\rel}-1)}.$$ Also, by definition, $w(d)=\frac{q^{d+1}-q^{d}}{q^{d+1}-1} \sim 1$. Thus, we have that $$\rho(\sigma,e_{\base},f_{\base};p) \sim \frac{1}{\perm(\sigma)\prod_{i=1}^m f_{i,\rel}q^{\sum_{i=1}^m (e_{i,\rel}-1)f_{i,\rel}}}=\frac{1}{\perm(\sigma)\prod_{i=1}^m f_{i,\rel}q^{\sum_{i=1}^m (e_{i,\rel}-1)f_{i,\rel}}}.$$
\end{proof}

\section{Future work}

We have implemented many formulas in this paper in Sagemath. In particular, we have implemented the code that computes $G$ and $\rho$. This is available \href{https://github.com/johng23/Density-of-p-adic-polynomials}{here}. We used the code to compute $\rho$ for many different combinations of inputs of $\sigma$, $e_{\base}$, and $f_{\base}$. The output satisfies the functional equation $\rho(\sigma, e_{\base},f_{\base};p) = \rho(\sigma,e_{\base},f_{\base};1/p)$, which gives further evidence to Conjecture 1.2 in \cite{bhargava_cremona_fisher_gajović_2022}. Furthermore, we have made the following more general observation. 

\begin{conj}
Define $$\rho(\sigma,e_{\base},f_{\base};p,t)=\frac{w(d)\sum_{A \subset [m]}G(A,0)(t)G(A,1)(t)}{\perm(\sigma)q^{\sum_{i=1}^m (e_i-1)f_i/2}\prod_{i=1}^m f_i}.$$ Then, $$\rho(\sigma,e,f;p,t) = \rho(\sigma,e,f;p^{-1},t^{-1}).$$
\end{conj} This conjecture is a generalization of Conjecture 1.2 of \cite{bhargava_cremona_fisher_gajović_2022}. Indeed, it reduces to Conjecture 1.2 after specializating $t$ to $q^{-e_{\base}/2}$.

In subsequent work, Asvin G, Yifan Wei, and I prove Theorem \ref{thm:main} for $\rho$ along with the functional equation $\rho(p)=\rho(1/p)$ and resolve \cite[Conjecture 1.2]{bhargava_cremona_fisher_gajović_2022} for $\rho$ in the tame case using a different approach via $p$-adic integration. In fact, our work will show a more general phenomenon than \cite[Conjecture 1.2]{bhargava_cremona_fisher_gajović_2022} about Galois maps between varieties over local fields. With a minor modification of our methods, we will also resolve the above conjecture.
\printbibliography

@misc{bhargava_cremona_fisher_gajović_2022, title={The density of polynomials of degree $n$ over $\mathbb{Z}_p$ having exactly $r$ roots in $\mathbb{Q}_p$}, url={https://arxiv.org/abs/2101.09590v4}, journal={arXiv.org}, author={Bhargava, Manjul and Cremona, John and Fisher, Tom and Gajović, Stevan}, year={2022}, month={Mar}}

@book{hasse_zimmer_1980, place={Berlin a.o}, title={Number theory}, publisher={Springer}, author={Hasse, Helmut and Zimmer, Horst G.}, year={1980}}

@article{dembo2002random,
  title={Random polynomials having few or no real zeros},
  author={Dembo, Amir and Poonen, Bjorn and Shao, Qi-Man and Zeitouni, Ofer},
  journal={Journal of the American Mathematical Society},
  volume={15},
  number={4},
  pages={857--892},
  year={2002}
}

@article{nguyen2021random,
  title={Random polynomials: central limit theorems for the real roots},
  author={Nguyen, Oanh and Vu, Van},
  journal={Duke Mathematical Journal},
  volume={170},
  number={17},
  pages={3745--3813},
  year={2021},
  publisher={Duke University Press}
}

@article{evans2006expected,
  title={The expected number of zeros of a random system of p -adic polynomials},
  author={Evans, Steven},
  journal={Electronic Communications in Probability},
  volume={11},
  pages={278--290},
  year={2006},
  publisher={Institute of Mathematical Statistics and Bernoulli Society}
}

@inproceedings{caruso2022zeroes,
  title={Where are the zeroes of a random p-adic polynomial?},
  author={Caruso, Xavier},
  booktitle={Forum of Mathematics, Sigma},
  volume={10},
  year={2022},
  organization={Cambridge University Press}
}

@article{buhler2006probability,
  title={The probability that a random monic p-adic polynomial splits},
  author={Buhler, Joe and Goldstein, Daniel and Moews, David and Rosenberg, Joel},
  journal={Experimental Mathematics},
  volume={15},
  number={1},
  pages={21--32},
  year={2006},
  publisher={Taylor \& Francis}
}

@article{kulkarni2021p,
  title={p-adic Integral Geometry},
  author={Kulkarni, Avinash and Lerario, Antonio},
  journal={SIAM Journal on Applied Algebra and Geometry},
  volume={5},
  number={1},
  pages={28--59},
  year={2021},
  publisher={SIAM}
}

@article{dokchitser_dokchitser_maistret_morgan_2022, title={Arithmetic of hyperelliptic curves over local fields}, DOI={10.1007/s00208-021-02319-y}, journal={Mathematische Annalen}, author={Dokchitser, Tim and Dokchitser, Vladimir and Maistret, Céline and Morgan, Adam}, year={2022}}

@misc{dokchitser_maistret_2020, title={Parity conjecture for abelian surfaces}, url={https://arxiv.org/abs/1911.04626}, journal={arXiv.org}, author={Dokchitser, Vladimir and Maistret, Celine}, year={2020}, month={Oct}}
\end{document}